\documentclass[11pt, a4paper]{amsart}
\usepackage{amsmath,amsfonts,amsthm}
\usepackage{mathdots}
\usepackage{pict2e}
\usepackage{tikz}

\newcommand{\cercle}{\circle{0.1}}
\newcommand{\segment}{\line(1,0){0.9}}
\newcommand{\segmentvertical}{\line(0,1){0.9}}

\newtheorem{theorem}{Theorem}[section]
\newtheorem{lemma}{Lemma}[section]
\newtheorem{proposition}{Proposition}[section]
\newtheorem{corollary}{Corollary}[section]

\newcounter{exercicecnt}
\newenvironment{exercice}
{\refstepcounter{exercicecnt}\par\noindent \theexercicecnt. }
{}

\title{Linearly recursive sequences and Dynkin diagrams}

\author{Christophe Reutenauer}
\address{Christophe Reutenauer:
D\'epartement de math\'ematiques, Universit\'e du Qu\'ebec \`a Montr\'eal}
\email{reutenauer.christophe@uqam.ca}

\date{\today}

\begin{document}  
 
\maketitle
\begin{abstract}
Motivated by a construction in the theory of cluster algebras (Fomin and Zelevinsky), one associates to each acyclic directed graph a family of sequences of natural integers, one for each vertex; this construction is called a {\em frieze}; these sequences are given by nonlinear recursions (with division), and the fact that they are integers is a consequence of the Laurent phenomenon of Fomin and Zelevinsky. If the sequences satisfy a linear recursion with constant coefficients, then the graph must be a Dynkin diagram or an extended Dynkin diagram, with an acyclic orientation. The converse also holds: the sequences of the frieze associated to an oriented Dynkin or Euclidean diagram satisfy linear recursions, and are even $\mathbb N$-rational. One uses in the proof objects called $SL_2$-{\em tilings of the plane}, which are fillings of the discrete plane such that each adjacent 2 by 2 minor is equal to 1. These objects, which have applications in the theory of cluster algebras, are interesting for themselves. 
Some problems, conjectures and exercises are given.

\end{abstract}

\section{Introduction}

The first scope of the present chapter is to prove a theorem which characterizes a class of linear recursive sequences associated to certain {\em quivers} (directed graphs): in general these sequences satisfy nonlinear recursions; it turns out that these sequences satisfy also a linear recursion exactly when the undirected underlying graph is of Dynkin, or extended Dynkin, type.

The secondary scope of the chapter is to introduce the reader to the notion of $SL_2$-tilings and their applications. An $SL_2$-tiling of the plane is a filling of the discrete plane by numbers, or elements of a commutative ring, in such a way that each adjacent 2 by 2 minor is equal to 1. 

Similar objects were considered by Coxeter \cite{coxeter}, Coxeter and Conway \cite{conway}, Di Francesco \cite{D}.

A remarkable subclass is obtained by prescribing the value 1 in a two-sided infinite discrete path of the discrete plane; then it extends uniquely to a $SL_2$-tiling, that may be easily computed by a well-known matrix representation. If the path is periodic, then the sequences on any discrete half-line satisfy linear recursions, and are even $\mathbb N$-rational. This allows to prove that the sequences associated to a frieze of type $\tilde{\mathbb A}$ are $\mathbb N$-rational.

The author is grateful to Pierre Auger, Val\'erie Berth\'e and Bernhard Keller for useful mail exchanges.

\section{$SL_2$-tilings of the plane}\label{plane}

Following \cite{ARS}, we call {\em $SL_2$-tiling of the plane} a mapping $t:\mathbb Z^2\mapsto K$, for some field $K$, such that for any $x,y$ in $\mathbb Z$, $$\left | \begin{array}{ll} t(x,y) & t(x+1,y)\\ t(x,y+1) & t(x+1,y+1) \end{array} \right | = 1.$$ Here we represent the discrete plane $\mathbb Z^2$, so that the $y$-axis points downwards, and the $x$-axis points to the right: see Figure \ref{coord}.

\begin{figure}[ht]\setlength{\unitlength}{4mm}
\begin{picture}
(12,12)(7,-9)
\put(14,-4){$P = (x,y)$}
\put(8,-3.5){${y}$}
 \put(13.5,2.5){${x}$} 
 \put(9,2){{\vector(0,-1){10}}}
 \put(9,2){{\vector(1,0){10}}}
 \end{picture}
\caption{Coordinate convention}\label{coord}
\end{figure}
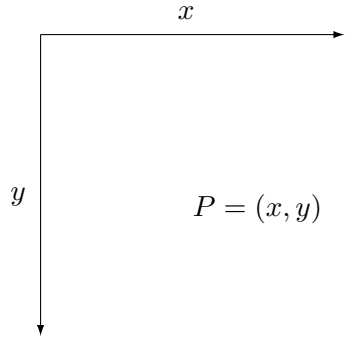

An example with $K=\mathbb Q$ is given in Figure \ref{tilingN}; it has values in $\mathbb N$. Another example, with $K$ the field of fractions over $\mathbb Q$ in the variables $a,b,c,d,e,f,\ldots$, ids given in Figure \ref{tilingabcdef}.

\begin{figure}
$$
\begin{array}{lllllllllllllll}
&&&&&&&&&& 1& 1& 1& 1& \\
&&&&&&&&&\bf 1& 1&2&3&4& \\
&&&&&&&&&\bf 1&2&5&8&11& \\
&&&...&&&&&2&\bf 1&3&8&13&18& \\
&&&&&&&& 1& 1&2^2&11&18&25& ...\\
&&&&&&&3& 1&2&9&5^2&41&57& \\
&&&&&&3&2&1&3&14&39&8^2&89& \\
& & & 1& 1& 1& 1& 1& 1&4&19&53&87&11^2& \\
 1& 1& 1& 1&2&3&4&5&6&5^2&119&332&545&758& \\
 1&2&3&4&9&14&19&24&29&121&24^2&1607&2368&3669& \\
&&&&&&&&&...&&&&& \\
\end{array}
$$
\caption{An $SL_2$-tiling over $\mathbb N$}
\label{tilingN}
\end{figure}

\begin{figure}
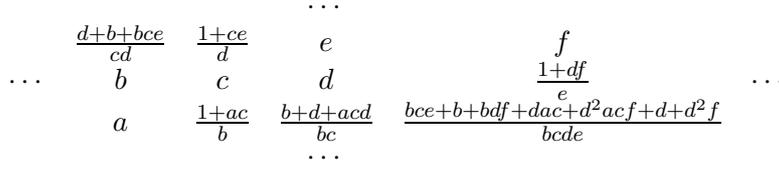

$$
\begin{array}{ccccccccc}
&&&\cdots&&\\
&\frac{d+b+bce}{cd}&\frac{1+ce}{d} & e& f& \\
\cdots&b&c&d& \frac{1+df}{e}&\cdots\\
&a&\frac{1+ac}{b} & \frac{b+d+acd}{bc}& \frac{bce+b+bdf+dac+d^2acf+d+d^2f}{bcde} \\
&&&\cdots&&
\end{array}
$$
\caption{An $SL_2$-tiling over ${\mathbb N}[a,b,c,\ldots,a^{-1},b^{-1},c^{-1},\ldots]$}
\label{tilingabcdef}
\end{figure}

These objects are an extension to the whole plane of the {\em frieze patterns} introduced by Coxeter \cite{coxeter} and studied by Conway and Coxeter \cite{conway}. Below is such a frieze-pattern, partially represented (it extends diagonally infinitely in both directions north-west and south-east). Note that Conway and Coxeter represent them horizontally, instead of diagonally as here.

$$
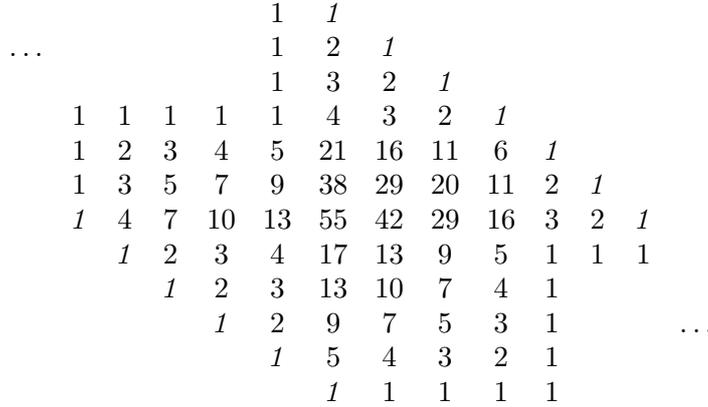
\begin{array}{cccccccccccccccc}
&&&&&1&\it 1 \\
\ldots&&&&&1&2&\it 1 \\
&&&&&1&3&2&\it 1 \\
&1&1&1&1&1&4&3&2&\it 1 \\
&1&2&3&4&5&21&16&11&6&\it 1 \\
&1&3&5&7&9&38&29&20&11&2&\it 1 \\
&\it 1&4&7&10&13&55&42&29&16&3&2&\it 1 \\
&&\it 1&2&3&4&17&13&9&5&1&1&1 \\
&&&\it 1&2&3&13&10&7&4&1 \\
&&&&\it 1&2&9&7&5&3&1&&&\ldots \\
&&&&&\it 1&5&4&3&2&1 \\
&&&&&&\it 1&1&1&1&1 
\end{array}
$$

An $SL_2$-tiling of the plane, viewed as an infinite matrix, has necessarily rank at least 2. Following \cite{BeR}, we say that the tiling is {\it tame} if its rank is 2. 

\section{$SL_2$-tiling associated to a bi-infinite discrete path}

Roughly speaking, a frontier is a discrete path (with steps that go from west to east, or south to north), which is infinite in both directions, and such that each vertex is labelled with a nonzero element of the ground field; see Figure \ref{frontier}. In most cases considered here, these elements will be all equal to 1, so that the frontier is simply a bi-infinite discrete path. 

The formal definition goes as follows. A {\em frontier} a bi-infinite sequence 
\begin{equation} \label{frontiervariable}
\ldots x_{-2}a_{-2}x_{-1}a_{-1}x_0a_0x_1a_1x_2a_2x_3a_3 \ldots
\end{equation}
where $x_i\in \{x,y\}$ and $a_i$ are elements of $K^*$, for any $i\in \mathbb Z$. It is called {\em admissible} if none of the two sequences $(x_n)_{n\geq0}$ and $(x_n)_{n\leq0}$ is ultimately constant. The $a_i$'s are called the {\em variables} of the frontier. Each frontier may be embedded into the plane: the variables label integer points in the plane, and the $x$ (resp. $y$) determine a bi-infinite discrete path, in such a way that $x$ (resp. $y$) corresponds to a segment of the form $[(a,b),(a+1,b)]$ (resp $[(a,b),(a,b-1)]$). Note the $b-1$: this is because of our coordinate conventions, see Figure \ref{coord} \footnote{The reader may consider that these conventions are not natural; but each choice has advantages and inconveniences.}.

An example is given in Figure \ref{tilingN}, with all the variables equal to 1, and where the embedding is represented by the 1's. In Figure \ref{tilingabcdef}, the frontier is coded by the word $...aybxcxdyexf...$.

Given an admissible frontier, embedded in the plane as explained previously, let $(u,v)\in \mathbb Z^2$. Then we obtain a finite word, which is a factor of the frontier, by projecting the point $(u,v)$ horizontally and vertically onto the frontier. We call this word the {\em word} of $(u,v)$. It is illustrated in Figure \ref{tilingabcdef}: the word associated to the point labelled $\frac{b+d+acd}{bc}$ is $aybxcxd$. Another example is shown in Figure \ref{frontier}, where the word associated to $P$ is $a_{-3}ya_{-2}ya_{-1}ya_{0}xa_1xa_2ya_3xa_4$.

\begin{figure}
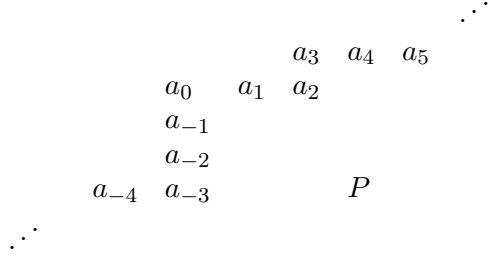

$$
 \begin{array}{lllllllllllllllll} 
&&&&&&&&&&\iddots    \\
&&&&&&&a_3&a_4&a_5 \\
&&&&&a_0&a_1&a_2  &\\
&&&&&a_{-1}  &&&\\
&&&&&a_{-2}  &&&\\
&&&&a_{-4}&a_{-3} &&&P\\
&&\iddots \end{array}
$$
\caption{A frontier}\label{frontier}
\end{figure}

\begin{theorem}\label{tiling-variables}
Given an admissible frontier, there exists a unique tame $SL_2$-tiling  $t$ of the plane over $K$, extending the embedding of the frontier into the plane. It is defined, for any point $(u,v)$ below the frontier, with associated word $a_0x_1a_1x_2...x_{n+1}a_{n+1}$, where $n\geq 1$ and $x_i\in\{x,y\}$, by the formula
\begin{equation}\label{tilingformulvariable}
t(u,v)=\frac{1}{a_1a_2...a_n} (1,a_0)\mu(a_1,x_2,a_2)\cdots \mu(a_{n-1},x_{n},a_{n})(1,a_{n+1})^t.
\end{equation}
\end{theorem}

Here we have used the following notation:
$$\mu(a,x,b)=\left(\begin{array}{cc}a&1\\0&b\end{array}\right), \,\, \,\mu(a,y,b)=\left(\begin{array}{cc}b&0\\1&a\end{array}\right).$$

\section{Proof of Theorem \ref{tiling-variables}}
We prove the theorem in the particular case where all variables on the frontier are equal to 1; we do not use more than this special case in the sequel of this article. The general case is proved in \cite{ARS}, Theorem 4, and tameness and uniqueness follows from \cite{BeR}, Proposition 7 and Proposition 12. 

In the special case we prove, one may omit the variables on the frontier, since they are all equal to 1; hence a frontier is simply a bi-infinite word on the alphabet $\{x,y\}$. Likewise, the word associated to each point is a finite word on this alphabet.

In this case, the theorem takes the following simpler form: let $\mu$ denote the homomorphism from the free monoid $\{x,y\}^*$ into the multiplicative group $SL_2(\mathbf Z)$ such that 
$$\mu(x)=\left(\begin{array}{cc}1&1\\0&1\end{array}\right) \quad \mbox{and}  \quad \mu(y)=\left(\begin{array}{cc}1&0\\1&1\end{array}\right).$$
Moreover, denote by $S(A)$ the sum of the coefficients of any matrix $A$. 
Then we to prove that: given an admissible frontier, with only 1's as variables, there exists a unique $SL_2$-tiling of the plane $t$ extending the embedding of the frontier into the plane. It is defined, for any point $(u,v)$ below the frontier, with associated word $ywx$, by the formula
$$t(u,v)=S(\mu(w)).
$$

Note that equivalently, if $(u,v)$ is a point with associated word $w$, one has
\begin{equation}\label{formula}
t(u,v)=\mu(w)_{22}.
\end{equation}
Indeed, we have $w=yw'x$, so that $$S(\mu(w'))=(1,1)\mu(w')\left(\begin{array}{c}1\\1\end{array}\right)$$
$$=
(\left(\begin{array}{cc}1&0\\1&1\end{array}\right)\mu(w')\left(\begin{array}{cc}1&1\\0&1\end{array}\right))_{22}=(\mu(y)\mu(w')\mu(x))_{22}=\mu(w)_{22}.$$

We prove below that formula (\ref{formula}) defines an $SL_2$-tiling $t$. Then, clearly $t(u,v)> 0$ for any $(u,v)\in \mathbb Z^2$. Then it is easily deduced, by induction on the length of the word associated to $(u,v)$, that $t(u,v)$ is uniquely defined by the $SL_2$ condition. This proves that the tiling is unique; note that in this special case tameness is not used to prove uniqueness. 

Now, we show that the function $t$ given by Eq.(\ref{formula}) is an $SL_2$-tiling of the plane. It is enough to show that for any $(u,v)\in \mathbb Z^2$, the determinant of the matrix $\left ( \begin{array}{ll} t(u,v) & t(u+1,v)\\ t(u,v+1) & t(u+1,v+1) \end{array} \right )$ is equal to 1.

By inspection of Figure \ref{words}, where $k,l\ge 0$ and $w=x_1\cdots x_n$, $n\geq 0$ and $x_i\in\{x,y\}$,
it is seen that the words associated to the four points $(u,v)$, $(u+1,v)$, $(u,v+1)$ and $(u+1,v+1)$ are  respectively of the form $w$, $wy^{l}x$, $yx^{k}w$ and $yx^{k}wy^{l}x$.
\begin{figure}
$$
\begin{array}{cccccccccccc}
&&&&&&&&\bf 1&\bf 1 \\
&&&&&&&&\vdots&  \\
&&&&&&&&\vdots&  \\
&&&&&&&&\bf 1&  \\
&&&&&&&\iddots&&\\
&&&&&&w&&&&  \\
&&&&&\iddots&&&&&    \\
&&&&&&&&&&  \\
\bf 1&\ldots&\ldots&\bf 1&&&&&(u,v)& (u+1,v)\\
\bf 1&&&&&&&&(u,v+1)&(u+1,v+1)
\end{array}
$$
\caption{}\label{words}
\end{figure}

Let $M=\mu(w)$. 
Then $t(u,v)=M_{22}=(0,1)M\left(\begin{array}{cc}0\\1\end{array}\right)$, $t(u+1,v)=(0,1)M\mu(y)^l\mu(x)\left(\begin{array}{cc}0\\1\end{array}\right)$, $t(u,v+1)=(0,1)\mu(y)\mu(x)^kM\left(\begin{array}{cc}0\\1\end{array}\right)$ and moreover $t(u+1,v+1)=(0,1)\mu(y)\mu(x)^kM\mu(y)^l\mu(x)\left(\begin{array}{cc}0\\1\end{array}\right)$.

We have to show that $t(u,v)t(u+1,v+1)-t(u,v+1)t(u+1,v)=1$. Equivalently that the matrix 
$$
\left(\begin{array}{cc}\lambda M\gamma&\lambda M\gamma'\\\lambda' M\gamma&\lambda' M\gamma'\end{array}\right),
$$
where $\lambda=(0,1)$, $\lambda'=(0,1)\mu(yx^k)$ and similarly for $\gamma,\gamma'$, has determinant 1. Now, this matrix is equal to the product
$$
\left(\begin{array}{cc}\lambda\\\lambda'\end{array}\right)M(\gamma,\gamma').
$$
The matrix $M=\mu(w)$ has determinant 1. Moreover, $\lambda'=(0,1)\mu(y)\mu(x^k)=(0,1)\left(\begin{array}{cc}1&0\\1&1\end{array}\right)\mu(x)^k=(1,1)\left(\begin{array}{cc}1&k\\0&1\end{array}\right)=(1,k+1)$. Thus  $\left(\begin{array}{cc}\lambda\\\lambda'\end{array}\right)=\left(\begin{array}{cc}0&1\\1&k+1\end{array}\right)$, which has determinant $-1$; similarly $(\gamma,\gamma')$ has determinant is $-1$, which ends the proof, except for tameness, which will be proved in Subsection \ref{tameness}.

We have also proved

\begin{corollary}\label{t(P)=mu(w)}
Let $t$ be the $SL_2$-tiling associated to some frontier with variables all equal to 1. Then for each point $P$ with associated word $w$, one has
$t(P)=\mu(w)_{22}$.
\end{corollary}

\section{$\mathbb N$-rational sequences}

\subsection{Equivalent definitions}

A series $S=\sum_{n\in \mathbf N} a_n x^n \in \mathbb N[[x]]$ is called {\em $\mathbb N$-rational} if it satisfies one of the two equivalent conditions (this equivalence is a particular case of the Kleene-Sch\"{u}tzenberger theorem, see \cite{BR}, Theorem I.7.1):

(i) for some matrices $\lambda \in \mathbb N^{1\times d}$, $M\in \mathbb N^{d\times d}$, $\gamma \in \mathbb N^{d\times 1}$, one has: $\forall n\in\mathbb N$, $a_n=\lambda M^n \gamma$. 

(ii) $S$ belongs to the smallest subsemiring of $\mathbb N[[x]]$ containing $\mathbb N[x]$ and closed under the operation $T\rightarrow T^*=\sum_{n\in \mathbf N} T^n$ (which is defined if $T$ has zero constant term);

We then say that the sequence $(a_n)$ is $\mathbb N$-{\em rational}. We do not prove this equivalence, since this is somewhat out of the scope of the present article. We shall use only form (i).
A third equivalence is the following: 

(iii) there is exists a rational, or equivalently, by Kleene's theorem, recognizable language $L$, such that
$a_n$ is for each $n$ the number of words of length $n$ in $L$ (see \cite{BR} Proposition 3.2.1).

Or equivalently:

(iv) there exist a finite directed graph $G$, a distinguished vertex $v_0$ and a subset $V_f$ of the set of vertices of $G$ such that for any $n$, $a_n$ is the number of paths in $G$ which go from $v_0$ to some vertex in $V_f$.

A simple consequence of (i) is that each $\mathbb N$-rational sequence $(a_n)$ satisfies a linear recursion with integer coefficients. Indeed, let $t^d-\alpha_1 t^{d-1}-\cdots -\alpha_d$ be the characteristic polynomial of $M$. Then by the Cayley-Hamilton theorem, we have $M^d=\alpha_1 M^{d-1}+\cdots +\alpha_d$; multiplying at the left by $\lambda M^n$ and at the right by $\gamma$, we therefore obtain that for any natural number $n$, $a_{n+d}=\alpha_1 a_{n+d-1}+\cdots +\alpha_d a_n$. Hence, $(a_n)$ satisfies the linear recursion associated to the characteristic polynomial of $M$.

Recall the well-known result that a sequence $(a_n)$ satisfies a linear recursion with constant coefficients if and only if the associated series $S=\sum_{n\in \mathbf N} a_n x^n$ is rational, that is, is the expansion of some rational function. We call such sequences {\em rational}.

We therefore deduce that each $\mathbb N$-rational sequence is rational and has coefficients in $\mathbb N$. The converse is however not true: see 4.3. 

It is generally believed that if a rational sequence is the counting sequence of some mathematical object, then it is actually $\mathbb N$-rational; this metamathematical principle, that goes back to Sch\"utzenberger, is illustrated by many examples, as Hilbert series of rings or monoids, and generating series of combinatorial structures; see the introduction of \cite{R} and \cite{BM} Section 2. An example within the present article are the friezes: if they are rational, they must be $\mathbb N$-rational. 

\subsection{Linear algebraic characterization and closure properties}

\begin{proposition}
The following conditions are equivalent:

(1) the sequence $(a_n)$ is $\mathbb N$-rational;

(2) there exists a finitely generated submodule of the $\mathbb N$-module of all sequences over $\mathbb N$ which contains $(a_n)$ and which is closed under the shift which maps each sequence $(b_n)$ onto the sequence $(b_{n+1})$;

(3) there exists a finitely generated $\mathbb N$-module ${\bf M}$, an element $V$ of ${\bf M}$, an $\mathbb N$-linear endomorphism $T$ of ${\bf M}$  and an $\mathbb N$-linear mapping $\phi:{\bf M}\rightarrow \mathbb N$ such that for any $n$, $a_n=\phi\circ T^n(V)$.
\end{proposition}

Recall that a module ${\bf M}$ over a semiring (here $\mathbb N$) is a commutative monoid ${\bf M}$ with a left action of the semiring on ${\bf M}$, and with the same axioms as those for modules over rings, except that one has to add the axiom $0.V=0$, for any $V\in {\bf M}$ (see \cite{BR} 5.1 for some details).

\begin{proof}
(1) implies (2): take $(a_n)$ in the from (i) above. Then define $a^{(i)}_n=e_iM^n \gamma$, where $(e_i)$ is the canonical basis of $\mathbb N^{1\times d}$. Let $\bf M$ be the $\mathbb N$-submodule of the module of all sequences over $\mathbb N$ generated by the sequences $a^{(i)}_n$, $i=1,\ldots ,d$. Then $a_n=\lambda M^n \gamma=\sum_{1\leq i\leq d}\lambda_i e_iM^n\gamma=\sum_{1\leq i\leq d}\lambda_i a^{(i)}_n$. Thus $(a_n)$ is in $\bf M$. Moreover, a similar calculation shows that the shifted sequence of  $a^{(i)}_n$, which is $a^{(i)}_{n+1}=\lambda'M^n\gamma$, with $\lambda'=e_iM$, is also in $\bf M$. This proves that $\bf M$ is closed under the shift, since the latter is $\mathbb N$-linear.

(2) implies (3): we take as module $\bf M$ the submodule given in (2), as $V$ the sequence $(a_n)$ itself, as $T$ the shift and as $\phi$ the mapping which sends each sequence onto its term of rank 0. Then clearly $a_n=\phi\circ T^n(V)$.

(3) implies (1): let $V_j$ be the $d$ generators of $\bf M$; define $M$ by the formula $T(V_i)=\sum_{j} M_{ij}V_j$. Then we have for any natural number $n$, $T^n(V_i)=\sum_{j}M^n_{ij}V_j$. Indeed, this is true for 
$n=0,1$ and we admit it for $n$. Then 
$$T^{n+1}(V_i)=T(T^n(V_i))=T(\sum_{j}M^n_{ij}V_j)=\sum_{j}M^n_{ij}T(V_j)$$
$$=\sum_{j}M^n_{ij}\sum_{k} M_{jk}V_k=\sum_{k}(\sum_{j}M^n_{ij}M_{jk})V_k=\sum_{k}M^{n+1}_{ik}V_k.$$ Define $\lambda\in \mathbb N^{1\times d}$ by $V=\sum_i\lambda_i V_i$ and $\gamma\in \mathbb N^{d\times 1}$ by $\gamma_j=\phi(V_j)$.

Then 
$$a_n=\phi\circ T^n(V)=\phi(\sum_i\lambda_iT^n(V_i))=\phi(\sum_i\lambda_i\sum_jM^n_{ij}V_j)$$
$$=\sum_{ij}\lambda_iM^n_{ij}\phi(V_j)=\sum_{ij}\lambda_iM^n_{ij}\gamma_j=\lambda M^n \gamma.
$$ 
Thus (1) holds.
\end{proof}

\begin{corollary}\label{prop}
(i) If $(a_n)$ and $(b_n)$ are two sequences such that for some integer $k$, the sequences $(a_n)$ and $(b_n)$ are ultimately equal, then they are simultaneously $\mathbb N$-rational or not.

(ii) If for some natural number $p$, the $p$ sequences $(a^{(i)}_n)$, $i=0,\ldots,p-1$, are $\mathbb N$-rational, then so is the sequence $a_n$ defined by $a_{i+np}=a^{(i)}_n$, for any $n$ and $i$, $i=0,\ldots,p-1$.

(iii) If $(a_n)$ and $(b_n)$ are two $\mathbb N$-rational sequences, then so is the sequence $(a_nb_n)$.
\end{corollary}
The latter sequence is called the {\em Hadamard product} of the two sequences.
\begin{proof}
(i) and (ii) are easy consequences of condition (2) in the proposition. For (ii), one uses (2) also, by by taking the Hadamard product of the two submodules. 
\end{proof}

In case (ii) of the Corollary, we say that $(a_n)$ is the {\em merge} (also called {\em interlacing}) of the 
sequences $(a^{(i)}_n)$, $i=0,\ldots, p-1$.

\subsection{Exponential polynomial and theorem of Berstel-Soittola}

Recall that each rational sequence $(a_n)$ may be uniquely expressed as an {\em exponential polynomial}, of the form
\begin{equation} \label{exponential_polynomial}
a_n=\sum_{i=1}^{k} P_i(n)\lambda_i^n,
\end{equation}
for $n$ large enough, where $P_i(n)$ is a nonzero polynomial in $n$ and the $\lambda_i$ are distinct nonzero complex numbers; see e.g. \cite{BR} 6.2. We call the $\lambda_i$ the {\em eigenvalues}
of the sequence $a_n$. We say that $a_n$ has a {\em dominating eigenvalue} if among the $\lambda_i$, there is a unique one of maximum modulus; for convenience, we include in this class the sequences ultimately equal to 0.

\begin{theorem}\label{merge}
A sequence is $\mathbb N$-rational if and only if it is a merge of rational sequences over $\mathbb N$ having a dominating eigenvalue.
\end{theorem}

This result is due to Berstel for the "only if" part, and to Soittola for the converse. A complete proof may found in \cite{BR}, Chapter 8. Here we need only Berstel's theorem.

It is a consequence of the theorem of Berstel that there exist rational sequences over $\mathbb N$ which are not $\mathbb N$-rational: see \cite{E} Theorem VI.6.1 and Example VI.6.1, or \cite{BR} Theorem 8.1.1 and Exercise 8.1.2. 

\subsection{An asymptotic lemma}

Given two sequences of positive real numbers $(a_n)$ and $(b_n)$, we shall write as usually $a_n\sim b_n$ if the quotient $a_n/b_n$ tends to 1 when $n$ tends to $\infty$. We also write 
$a_n\approx b_n$ to express the fact that for some positive constant $C$, one has $\lim_{k\rightarrow \infty}a_n/b_n=C$. 
Clearly $a_n\sim b_n$ implies $a_n\approx b_n$. 

\begin{corollary}\label{equivalent}
Let $a(v,n)$ be a family, indexed by the finite set $V$, of unbounded $\mathbb N$-rational sequences of positive integers. There exist an integer $p\geq 1$, real numbers $\lambda(v,i)\geq 1$ and integers $e(v,i)\geq 0$, for $v\in V$ and $i=0,...,p$, such that:

(i) for every  $v\in V$ and every $i=0,...,p, a(v,pn+i)\approx\lambda(v,i)^{n} n^{e(v,i)}$;

(ii) for every  $v\in V$, there exists $i=0,...,p$ such that $\lambda(v,i)>1$ or $e(v,i)\geq 1$;

(iii) for every $v\in V$, $\lambda (v,0)=\lambda(v,p)$ and $e(v,0)=e(v,p)$.
\end{corollary}

\begin{proof}
Let $(a_n)$ be an $\mathbb N$-rational sequence having a dominating eigenvalue, with $a_n>0$. We have then Eq. (\ref{exponential_polynomial}) with $\mid \lambda_1\mid > \mid\lambda_2\mid,\ldots, \mid \lambda_k\mid$, and nonzero $P_i$'s. Let $\alpha$ be the dominating coefficient of $P_1$, $e$ the degree of the latter and $\lambda=\lambda_1$. Then $a_n\sim \alpha n^e\lambda^n$. Since the $a_n$ are positive natural numbers, and since $a_{n+1}/a_n\sim    
\lambda$, we must have $\lambda\in \mathbb R_+$. We cannot have $\lambda <1$, since otherwise, $a_n$ being an integer, $a_n=0$ for $n$ large enough. Thus $\lambda \geq 1$. If $a_n$ is unbounded, then either $e>0$ or $\lambda >1$.

Note that if a rational sequence is the merge of $p$ sequences having a dominating eigenvalue, this is true also for each multiple of $p$. Therefore, using Theorem \ref{merge}, we see that there exists some $p$ such that each series $a(v,n+pi)$, $i=0,\ldots, p-1$ is $\mathbb N$-rational and has a dominating eigenvalue. 

By the first part of the proof, we see that (i) holds for $i=0,\ldots, p-1$. Now, we have $a(v,pn+p)=a(v,p(n+1))\approx \lambda(v,0)^{n+1}(n+1)^{e(v,0)}$ by the $i=0$ case. Therefore $a(v,pn+p)\approx  \lambda(v,0)^{n}n^{e(v,0)}$: thus (i) holds also for $i=p$ and moreover (iii) holds.

The sequence $a(v,n)$ is unbounded; hence for some $i=0,\ldots,p-1$, the sequence $a(pn+i)$ is unbounded and therefore either $e(v,i)>0$ or $\lambda(v,i)>1$. Thus (ii) holds.

\end{proof}

\section{$\mathbb N$-rationality of the rays in $SL_2$-tilings} 

Given a mapping $t: \mathbb Z^2\rightarrow K$, a point $M\in \mathbb Z^2$ and a nonzero vector $V\in \mathbb Z^2$, we consider the sequence $a_n=t(M+nV)$. Such a sequence will be called a {\em ray associated to} $t$. We call $M$ the {\em origin} of the ray and $V$ its {\em directing vector}. The ray is {\em horizontal} if $V=(1,0)$, {\em vertical} if $V=(0,1)$ and {\em diagonal} if $V=(1,1)$ (cf. Figure \ref{coord}). 

\begin{theorem} \label{rational}
Suppose that the frontier in Theorem \ref{tiling-variables} is ultimately periodic and that each variable is equal to 1. Then each ray associated to $t$ is $\mathbf N$-rational.
\end{theorem}

\begin{proof}
We prove the theorem in the case where the directing vector $V=(a,b)$ satisfies $a,b\geq 0$. The other cases are left to the reader and will not be used in the sequel.

1. The points $M+nV$ are, for $n$ large enough, all above or all below the frontier, since the frontier is admissible and by the hypothesis on the directing vector. Since by Lemma \ref{prop} $\mathbb N$-rationality is not affected by changing a finite number of values, we may assume that they are all below. 

2. Let $w_n$ be the word associated to the point $M+nv$. 
By ultimate periodicity of the frontier, there exists an integer $q\geq 1$ and words $v_0,...,v_{q-1}$, $u'_0,...,u'_{q-1},u_0,...,u_{q-1}$ such that for any $i=0,...,q-1$ and for $n$ large enough, $w_{i+nq}={u'}_i^nv_iu_i^n$. This is seen by inspection.

3. It follows from Corollary \ref{t(P)=mu(w)} that for some $2\times 2$ matrices $M'_i,N_i,M_i$ over $\mathbf N$, one has for any $i=0,\ldots,q-1$ and $n$ large enough, $a_{i+nq} = \lambda_i {M'_i}^nN_iM_i^n\gamma_i$, where $\lambda_i\in \mathbf N^{1\times 2},\gamma_i\in \mathbf N^{2\times1}$.

4. Consequently, by Corollary \ref{prop}, the sequence $(a_n)$ is $\mathbb N$-rational. Indeed, a sequence of the form  $\lambda {M'}^nNM^n\gamma$ is a $\mathbf N$-linear combination of sequences, each of which is a Hadamard product of two $\mathbf N$-rational sequences. 
\end{proof}

In the case of a purely periodic frontier (this is the case of the $SL_2$-tilings associated to the friezes of type $\tilde {\mathbb A}_n$), there is a sharpening of this result; it explains why the linear recursions in these tilings are essentially of length two.

\begin{proposition}
Let $t$ be the $SL_2$-tiling associated to a purely periodic frontier, with variables equal to 1, associated to the admissible frontier ${}^\infty w^\infty$, for some word $w\in \{x,y\}^*$. Let $a$ (resp.$b$) be the number of $x$'s (resp. of $y$'s) in $w$, and let $p$ be the lowest common multiple of $a,b$. Then each diagonal ray is a merge of $p$ rational sequences, each of which satisfies the linear recursion associated to the characteristic polynomial of the matrix $\mu(w)^{p/a+p/b}$.
\end{proposition}

This result is illustrated below. We have $w=xxy$, $a=2$, $b=1$, $p=2$, $p/a+p/b=3$, $$\mu(w)=\left(\begin{array}{cc}1&1\\0&1\end{array}\right)^2 \left(\begin{array}{cc}1&0\\1&1\end{array}\right)=\left(\begin{array}{cc}3&2\\1&1\end{array}\right),$$
$$\mu(w)^3=\left(\begin{array}{cc}3&2\\1&1\end{array}\right)^3=\left(\begin{array}{cc}41&30\\15&11\end{array}\right).$$
The characteristic polynomial of the latter matrix is $t^2-52t+1$ and the associated linear recursion is $u_{n+2}=52u_{n+1}-1$. We see indeed in Figure \ref{tiling} that $1351=52.26-1$, $2131=52.41-1$ and $5042=52.97-2$.

\begin{figure}
$$\begin{array}{lllllllllllllllllllllllllllllllllll}
&&&&&&&&&&&&1&1\\
&&&&&&&&&&1&1&1&2 \\
&&&&&&&&1&1&1&2&3&7\\
&&&\ldots&&&1&\bf 1&\bf 1&\bf 2&3&7&11&26 \\
&&&&1&1&1&2&3&7&11&26&41&97 \\
&&1&1&1&2&3&7&11&\bf 26&\bf 41&\bf 97&153&362\\
1&1&1&2&3&7&11&26&41&97&153&362&571 &1351\\
1&2&3&7&11&26&41&97&153&362&571&\bf 1351&\bf  2131&\bf 5042\\
&&&&&&&&&\ldots
\end{array}
$$
\caption{an $SL_2$-tiling}\label{tiling}
\end{figure}

In order to prove the proposition, one has to mimick with more precision Part 2. of the proof of the theorem. Details are left to the reader.

\section{Friezes}\label{frieze}
\subsection{Numerical friezes}\label{numfr}
Let $Q$ be a quiver (a directed graph) with set of vertices $V$ and set of arrows $E$. One assumes that it is acyclic. Associate to each vertex $v$ the sequence $a(v,n)$ of numbers defined by the recursion: $a(v,0)=1$ and 
\begin{equation}\label{friserecursion}
\quad a(v,n+1)=\frac{1+\prod_{w\rightarrow v}a(w,n+1)\prod_{v\rightarrow w}a(w,n)}{a(v,n)}  .
\end{equation}

The family of sequences $a(v,n)$, one for each vertex $v$, is called the {\em frieze} associated to $Q$. The fact that the quiver is acyclic guarantees the existence and the unicity of the sequences. A priori these numbers are positive rational numbers. A remarkable fact, due to Fomin and Zelevinski, is a consequence of their {\em Laurent phenomenon} \cite{fominzel, fomin}: the sequences are natural integers. That is, in the above fraction, the denominator always divides the numerator. 
This will proved below.

The terminology "frieze" comes from the fact that one may take infinitely many copies of the quiver $Q$, one copy $Q_n$ for each natural integer, with corresponding vertices $v_n$; one adds arrows $v_n\rightarrow w_{n+1}$ if in $Q$ one has $w\rightarrow v$; then the labelling $a(v,n)$ of vertex $v_n$ is obtained recursively by the previous formula. It means that the label of a vertex $v_{n+1}$ is obtained as follows: take the product of the labels of the sources of the ingoing arrows, add 1, and divide by the label of $v_n$. For example, in Figure \ref{friseA2}, the value 26 is $7\times 11+1$ divided by 3.

These recursions, although highly nonlinear, produce sometimes rational, or even $\mathbb N$-rational, sequences. In this case, we say that the frieze is {\em rational}, or $\mathbb N$-{\em rational}. This is the case for example when the quiver has two vertices with two edges: then one obtains the Fibonacci numbers of even rank. 

The question of rationality of the frieze is the central question which will be answered in this article.

Let $G$ be an undirected graph without loops. We say that a {\em frieze is of type} $G$ if the quiver $Q$ on which it is constructed is obtained from $G$ by some acyclic orientation of its edges.

The friezes shown in the Figures \ref{friseA2}, \ref{friseD7} and \ref{friseE6} are of the type indicated, which is a graph defined in Section \ref{Dd}.

\setlength{\unitlength}{1cm}

\setlength{\textwidth}{17.59cm}
\setlength{\oddsidemargin}{-0.54cm}

\begin{figure}[ht]
\begin{center}
\begin{tikzpicture} [>=latex, scale=0.8]
\node at (0,0) (a) {$1$};
\node at (3,0) (b) {$2$};
\node at (6,0) (c) {$11$};
\node at (9,0) (d) {$97$};
\node at (12,0) (e) {$571$};

\node at (1.5,0.75) (f) {$1$};
\node at (4.5,0.75) (g) {$7$};
\node at (7.5,0.75) (h) {$41$};
\node at (10.5,0.75) (i) {$362$};
\node at (13.5,0.75) (j) {$2131$};

\node at (1.5,-0.75) (k) {$1$};
\node at (4.5,-0.75) (l) {$3$};
\node at (7.5,-0.75) (m) {$26$};
\node at (10.5,-0.75) (n) {$153$};
\node at (13.5,-0.75) (o) {$1351$};

\draw(a)[->] -- (f);
\draw(a)[->] -- (k);
\draw(f)[->] -- (b);
\draw(k)[->] -- (b);
\draw(b)[->] -- (g);
\draw(b)[->] -- (l);
\draw(g)[->] -- (c);
\draw(l)[->] -- (c);
\draw(c)[->] -- (h);
\draw(c)[->] -- (m);
\draw(h)[->] -- (d);
\draw(m)[->] -- (d);
\draw(d)[->] -- (i);
\draw(d)[->] -- (n);
\draw(i)[->] -- (e);
\draw(n)[->] -- (e);
\draw(e)[->] -- (j);
\draw(e)[->] -- (o);
\draw(k)[->] -- (f);
\draw(l)[->] -- (g);
\draw(m)[->] -- (h);
\draw(n)[->] -- (i);
\draw(o)[->] -- (j);
\draw [->] (f) .. controls (2.7, -0.75) .. (l);
\draw [->] (g) .. controls (5.7, -0.75) .. (m);
\draw [->] (h) .. controls (8.7, -0.75) .. (n);
\draw [->] (i) .. controls (11.7, -0.75) .. (o);
\end{tikzpicture}
\end{center}
\caption{A frieze of type $\widetilde{\mathbb{A}}_2$}\label{friseA2}
\end{figure}

\setlength{\unitlength}{1cm}

\setlength{\textwidth}{17.59cm}
\setlength{\oddsidemargin}{-0.54cm}

\newlength{\abscisse}
\setlength{\abscisse}{0cm}

\newlength{\ordonnee}
\setlength{\ordonnee}{0cm}

\newlength{\pashorizontal}
\setlength{\pashorizontal}{3cm}

\newlength{\pasvertical}
\setlength{\pasvertical}{1cm}

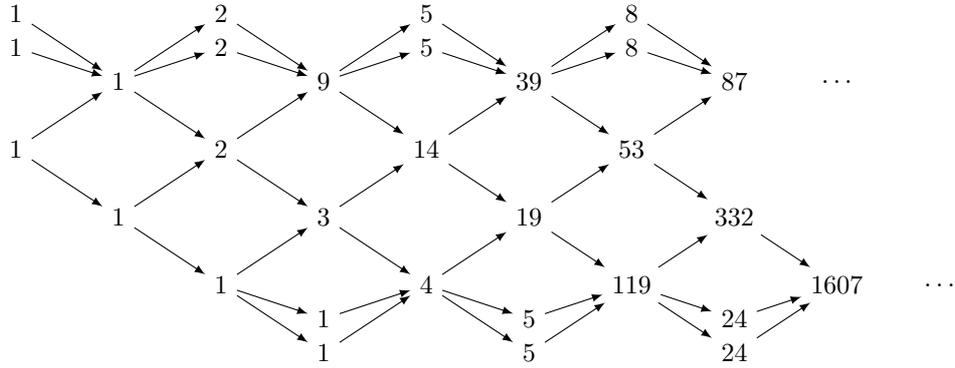
\begin{figure}[ht]
\begin{center}
\begin{tikzpicture} [>=latex, scale=0.9, transform shape]
\node at (\abscisse,\ordonnee) (a) {$1$};		\addtolength{\abscisse}{\pashorizontal}
\node at (\abscisse,\ordonnee) (b) {$9$};		\addtolength{\abscisse}{\pashorizontal}
\node at (\abscisse,\ordonnee) (c) {$39$};		\addtolength{\abscisse}{\pashorizontal}
\node at (\abscisse,\ordonnee) (d) {$87$};	\addtolength{\abscisse}{0.5\pashorizontal}
\node at (\abscisse,\ordonnee) (e) {$\ldots$};

\setlength{\abscisse}{-0.5\pashorizontal}
\setlength{\ordonnee}{\pasvertical}

\node at (\abscisse,\ordonnee) (aa) {$1$};	\addtolength{\abscisse}{\pashorizontal}
\node at (\abscisse,\ordonnee) (ba) {$2$};	\addtolength{\abscisse}{\pashorizontal}
\node at (\abscisse,\ordonnee) (ca) {$5$};		\addtolength{\abscisse}{\pashorizontal}
\node at (\abscisse,\ordonnee) (da) {$8$};

\setlength{\abscisse}{-0.5\pashorizontal}
\setlength{\ordonnee}{0.5\pasvertical}

\node at (\abscisse,\ordonnee) (aam) {$1$};	\addtolength{\abscisse}{\pashorizontal}
\node at (\abscisse,\ordonnee) (bam) {$2$};	\addtolength{\abscisse}{\pashorizontal}
\node at (\abscisse,\ordonnee) (cam) {$5$};		\addtolength{\abscisse}{\pashorizontal}
\node at (\abscisse,\ordonnee) (dam) {$8$};

\draw(aam)[->] -- (a);
\draw(bam)[->] -- (b);
\draw(cam)[->] -- (c);
\draw(dam)[->] -- (d);

\draw(aa)[->] -- (a);
\draw(ba)[->] -- (b);
\draw(ca)[->] -- (c);
\draw(da)[->] -- (d);

\draw(a)[->] -- (ba);
\draw(b)[->] -- (ca);
\draw(c)[->] -- (da);

\draw(a)[->] -- (bam);
\draw(b)[->] -- (cam);
\draw(c)[->] -- (dam);

\setlength{\abscisse}{-0.5\pashorizontal}
\setlength{\ordonnee}{-\pasvertical}

\node at (\abscisse,\ordonnee) (a1) {$1$};	\addtolength{\abscisse}{\pashorizontal}
\node at (\abscisse,\ordonnee) (b1) {$2$};	\addtolength{\abscisse}{\pashorizontal}
\node at (\abscisse,\ordonnee) (c1) {$14$};	\addtolength{\abscisse}{\pashorizontal}
\node at (\abscisse,\ordonnee) (d1) {$53$};

\draw(a1)[->] -- (a);
\draw(b1)[->] -- (b);
\draw(c1)[->] -- (c);
\draw(d1)[->] -- (d);

\draw(a)[->] -- (b1);
\draw(b)[->] -- (c1);
\draw(c)[->] -- (d1);

\setlength{\abscisse}{0cm}
\setlength{\ordonnee}{-2\pasvertical}

\node at (\abscisse,\ordonnee) (a2) {$1$};	\addtolength{\abscisse}{\pashorizontal}
\node at (\abscisse,\ordonnee) (b2) {$3$};	\addtolength{\abscisse}{\pashorizontal}
\node at (\abscisse,\ordonnee) (c2) {$19$};	\addtolength{\abscisse}{\pashorizontal}
\node at (\abscisse,\ordonnee) (d2) {$332$};

\draw(a1)[->] -- (a2);
\draw(b1)[->] -- (b2);
\draw(c1)[->] -- (c2);
\draw(d1)[->] -- (d2);

\draw(a2)[->] -- (b1);
\draw(b2)[->] -- (c1);
\draw(c2)[->] -- (d1);

\setlength{\abscisse}{0.5\pashorizontal}
\setlength{\ordonnee}{-3\pasvertical}

\node at (\abscisse,\ordonnee) (a3) {$1$};	\addtolength{\abscisse}{\pashorizontal}
\node at (\abscisse,\ordonnee) (b3) {$4$};	\addtolength{\abscisse}{\pashorizontal}	
\node at (\abscisse,\ordonnee) (c3) {$119$};	\addtolength{\abscisse}{\pashorizontal}
\node at (\abscisse,\ordonnee) (d3) {$1607$};	\addtolength{\abscisse}{0.5\pashorizontal}
\node at (\abscisse,\ordonnee) (e3) {$\ldots$};

\draw(a2)[->] -- (a3);
\draw(b2)[->] -- (b3);
\draw(c2)[->] -- (c3);
\draw(d2)[->] -- (d3);

\draw(a3)[->] -- (b2);
\draw(b3)[->] -- (c2);
\draw(c3)[->] -- (d2);

\setlength{\abscisse}{\pashorizontal}
\addtolength{\ordonnee}{-0.5\pasvertical}

\node at (\abscisse,\ordonnee) (a34) {$1$};	\addtolength{\abscisse}{\pashorizontal}
\node at (\abscisse,\ordonnee) (b34) {$5$};	\addtolength{\abscisse}{\pashorizontal}
\node at (\abscisse,\ordonnee) (c34) {$24$};	\addtolength{\abscisse}{\pashorizontal}

\draw(a3)[->] -- (a34);
\draw(b3)[->] -- (b34);
\draw(c3)[->] -- (c34);

\draw(a34)[->] -- (b3);
\draw(b34)[->] -- (c3);
\draw(c34)[->] -- (d3);

\setlength{\abscisse}{\pashorizontal}
\setlength{\ordonnee}{-4\pasvertical}

\node at (\abscisse,\ordonnee) (a4) {$1$};	\addtolength{\abscisse}{\pashorizontal}
\node at (\abscisse,\ordonnee) (b4) {$5$};	\addtolength{\abscisse}{\pashorizontal}
\node at (\abscisse,\ordonnee) (c4) {$24$};	\addtolength{\abscisse}{\pashorizontal}

\draw(a3)[->] -- (a4);
\draw(b3)[->] -- (b4);
\draw(c3)[->] -- (c4);

\draw(a4)[->] -- (b3);
\draw(b4)[->] -- (c3);
\draw(c4)[->] -- (d3);

\end{tikzpicture}
\end{center}
\caption{A frieze of type $\widetilde{D}_7$}\label{friseD7}
\end{figure}

\setlength{\abscisse}{0cm}

\newlength{\pas}
\setlength{\pas}{1.2cm}

\begin{figure}[ht]
\begin{center}
\begin{tikzpicture} [>=latex, scale=0.7,transform shape]
\node at (\abscisse,0) (a) {$1$}; 		\addtolength{\abscisse}{\pas}
\node at (\abscisse,0) (b) {$1$};		\addtolength{\abscisse}{\pas}
\node at (\abscisse,0) (c) {$1$};			\addtolength{\abscisse}{\pas}
\node at (\abscisse,0) (d) {$2$};		\addtolength{\abscisse}{\pas}
\node at (\abscisse,0) (e) {$3$};		\addtolength{\abscisse}{\pas}
\node at (\abscisse,0) (f) {$28$};		\addtolength{\abscisse}{\pas}
\node at (\abscisse,0) (g) {$2$};		\addtolength{\abscisse}{\pas}
\node at (\abscisse,0) (h) {$19$};		\addtolength{\abscisse}{\pas}
\node at (\abscisse,0) (i) {$245$};		\addtolength{\abscisse}{\pas}
\node at (\abscisse,0) (j) {$10$};		\addtolength{\abscisse}{\pas}	
\node at (\abscisse,0) (k) {$129$};		\addtolength{\abscisse}{\pas}
\node at (\abscisse,0) (l) {$8762$};		\addtolength{\abscisse}{\pas}
\node at (\abscisse,0) (m) {$13$};		\addtolength{\abscisse}{\pas}
\node at (\abscisse,0) (n) {$883$};		\addtolength{\abscisse}{\pas}
\node at (\abscisse,0) (o) {$78574$};	

\setlength{\abscisse}{2 \pas}
\addtolength{\abscisse}{-0.5\pas}

\node at (\abscisse,1) (ca) {$1$};		\addtolength{\abscisse}{3\pas}
\node at (\abscisse,1) (fa) {$3$};		\addtolength{\abscisse}{3\pas}
\node at (\abscisse,1) (ia) {$19$};		\addtolength{\abscisse}{3\pas}
\node at (\abscisse,1) (la) {$129$};		\addtolength{\abscisse}{3\pas}
\node at (\abscisse,1) (oa) {$883$};		\addtolength{\abscisse}{\pas}
\node at (\abscisse,1) (pa) {$\ldots$};

\setlength{\abscisse}{2 \pas}
\addtolength{\abscisse}{-0.5\pas}

\node at (\abscisse,-1) (cb) {$1$};		\addtolength{\abscisse}{3\pas}
\node at (\abscisse,-1) (fb) {$3$};		\addtolength{\abscisse}{3\pas}
\node at (\abscisse,-1) (ib) {$19$};		\addtolength{\abscisse}{3\pas}
\node at (\abscisse,-1) (lb) {$129$};		\addtolength{\abscisse}{3\pas}
\node at (\abscisse,-1) (ob) {$883$};		\addtolength{\abscisse}{\pas}
\node at (\abscisse,-1) (pb) {$\ldots$};

\setlength{\abscisse}{\pas}

\node at (\abscisse,2) (caa) {$1$};		\addtolength{\abscisse}{3\pas}
\node at (\abscisse,2) (faa) {$2$};		\addtolength{\abscisse}{3\pas}
\node at (\abscisse,2) (iaa) {$2$};		\addtolength{\abscisse}{3\pas}
\node at (\abscisse,2) (laa) {$10$};		\addtolength{\abscisse}{3\pas}
\node at (\abscisse,2) (oaa) {$13$};

\setlength{\abscisse}{\pas}

\node at (\abscisse,-2) (cbb) {$1$};		\addtolength{\abscisse}{3\pas}
\node at (\abscisse,-2) (fbb) {$2$};		\addtolength{\abscisse}{3\pas}
\node at (\abscisse,-2) (ibb) {$2$};		\addtolength{\abscisse}{3\pas}
\node at (\abscisse,-2) (lbb) {$10$};		\addtolength{\abscisse}{3\pas}
\node at (\abscisse,-2) (obb) {$13$};

\draw(a)[->] -- (b);
\draw(b)[->] -- (c);
\draw(c)[->] -- (d);
\draw(d)[->] -- (e);
\draw(e)[->] -- (f);
\draw(f)[->] -- (g);
\draw(g)[->] -- (h);
\draw(h)[->] -- (i);
\draw(i)[->] -- (j);
\draw(j)[->] -- (k);
\draw(k)[->] -- (l);
\draw(l)[->] -- (m);
\draw(m)[->] -- (n);
\draw(n)[->] -- (o);

\draw(ca)[->] -- (c);
\draw(fa)[->] -- (f);
\draw(ia)[->] -- (i);
\draw(la)[->] -- (l);
\draw(oa)[->] -- (o);

\draw(cb)[->] -- (c);
\draw(fb)[->] -- (f);
\draw(ib)[->] -- (i);
\draw(lb)[->] -- (l);
\draw(ob)[->] -- (o);

\draw(caa)[->] -- (ca);
\draw(faa)[->] -- (fa);
\draw(iaa)[->] -- (ia);
\draw(laa)[->] -- (la);
\draw(oaa)[->] -- (oa);

\draw(cbb)[->] -- (cb);
\draw(fbb)[->] -- (fb);
\draw(ibb)[->] -- (ib);
\draw(lbb)[->] -- (lb);
\draw(obb)[->] -- (ob);

\draw(c)[->] -- (fa);
\draw(f)[->] -- (ia);
\draw(i)[->] -- (la);
\draw(l)[->] -- (oa);

\draw(c)[->] -- (fb);
\draw(f)[->] -- (ib);
\draw(i)[->] -- (lb);
\draw(l)[->] -- (ob);

\draw(ca)[->] -- (faa);
\draw(fa)[->] -- (iaa);
\draw(ia)[->] -- (laa);
\draw(la)[->] -- (oaa);

\draw(cb)[->] -- (fbb);
\draw(fb)[->] -- (ibb);
\draw(ib)[->] -- (lbb);
\draw(lb)[->] -- (obb);

\setlength{\abscisse}{4\pas}
\addtolength{\abscisse}{-0.25\pas}

\setlength{\ordonnee}{0.5cm}

\draw [->] (c) .. controls (\abscisse, \ordonnee) .. (e.north); 	\addtolength{\abscisse}{3\pas}
\draw [->] (f) .. controls (\abscisse, \ordonnee) .. (h.north);		\addtolength{\abscisse}{3\pas}
\draw [->] (i) .. controls (\abscisse, \ordonnee) .. (k.north);		\addtolength{\abscisse}{3\pas}
\draw [->] (l) .. controls (\abscisse, \ordonnee) .. (n.north);		\addtolength{\abscisse}{3\pas}

\end{tikzpicture}
\end{center}
\caption{A frieze of type $\widetilde{\mathbb{E}}_6$}\label{friseE6}
\end{figure}
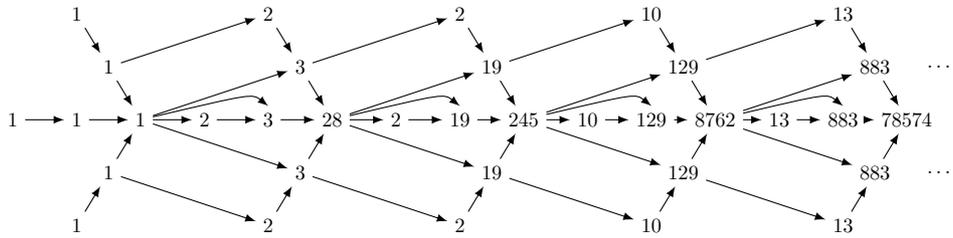

\subsection{Quiver mutations}

We consider a finite set $V$ and the ring of {\em Laurent polynomials} $L =\mathbb Z[v^{\pm 1}, v\in V]$, and the field of rational function $F=\mathbb Q(V)$. Note that $L$ is a factorial ring. Call {\em basis} a family $X=(x_v)_{v\in V}$ of rational functions which freely generates the field $F$; that is, $v\mapsto x_v$ is an automorphism of $F$. An example of basis is $V$ itself.

The quivers we consider here are without loops and without 2-cycles; they may have multiple arrows and are not assumed to be acyclic as in Subsection \ref{numfr}. Consider an {\em initial quiver} $Q_0$ with vertex set $V$, and label each vertex $v$ of it by $v$ itself, considered as an element of the field $F=K(V)$. We consider a quiver $Q$ with same vertex set $V$, with each $v\in V$ labelled by $x_v$, where $X$ is a basis as defined above; given some vertex $u$ of $Q$, we define the {\em mutation of} $Q$ {\em at} $u$, denoted $\mu_u$: it defines a new quiver $Q'$, with same vertex set $V$, with basis $Y$ such that $y_v=x_v$ if $v\neq u$ and with $y_ux_u=\prod_{v\rightarrow u}x_v +\prod_{u\rightarrow v} x_v$, where the arrows are of course taken with multiplicities. Moreover $Q'$ is obtained from $Q$ as follows:
\begin{enumerate}
\item for each arrows $v\rightarrow u\rightarrow w$, add an arrow $u\rightarrow w$ (transitive closure at $u$);
\item reverse each arrow incident at $u$ (reversion);
\item remove each pair of opposite arrows at $u$ until no such pair exists (removal of 2-cycles); 
\end{enumerate}

The polynomial $\prod_{v\rightarrow u}v +\prod_{u\rightarrow v} v \in {\mathbb Z}[V]$ is called the {\em mutation polynomial of $Q$ at} $u$; note that this polynomial does not depend on $u$. We denote $Q'=\mu_u(Q)$. One verifies that the mutation polynomial of $Q$ at $u$ is equal to the mutation polynomial of $\mu_u(Q)$ at $u$.

Note that $Q'$ is without loops, since $Q$ has no loops nor 2-cycles. Moreover, $Q'$ has no 2-cycles by construction. 

\begin{proposition}\label{mutprop}
(i) (involution) Mutation at $u$ is an involution.

(ii) (commutation) If $u,w$ are two vertices in $Q$ without arrow between them, then $\mu_u\mu_w(Q)=\mu_w\mu_u(Q)$.

(iii) (braid) Suppose that $u,w$ are vertices of $Q$. Let $A$ (resp. $B,C$) be the mutation polynomial of $Q$ at $u$ (resp. of $\mu_u(Q)$ at $w$, resp. of $\mu_w\mu_u(Q)$ at $u$). If there is no arrow in $Q$ between $u$ and $w$, then $A=C$ and they do not depend on $w$. If there is an arrow, then $C=mA\mid_{w\leftarrow B_0/w}$ for some Laurent monomial $m$, not depending on $u$, where $B_0$ is the monomial $B_0=B\mid_{u\leftarrow 0}$. 
\end{proposition}

Each quiver $Q$ is completely determined by the antisymmetric matrix $M\in{\mathbb Z}^{V\times V}$ defined by: $m_{uv}$ is the number of arrows $u\rightarrow v$ minus the number of arrows $v\rightarrow u$; note that at most one of these two numbers is nonzero, since $Q$ has no 2-cycles. If $Q'$ is obtained from $Q$ by mutation at $u$, then its matrix $M'$ satisfies, when $v\neq w$:
$m'_{vw}=-m_{vw}$ if $v=u$ or $w=u$ and 
$m'_{vw}=m_{vw}+sgn(m_{vu})max(m_{vu}m_{uw},0)$ if $v\neq u$ and $w\neq u$. We leave the verification of this fact to the reader.

\begin{proof}
(i) If, with the notations above, we apply a second time the mutation at $u$, we obtain a matrix $M''$. Suppose that $v=u$ or $w=u$; then $m''_{vw}=-m'_{vw}=m_{vw}$. If $v\neq u$ and $w\neq u$, then $$m''_{vw}=m'_{vw}+sgn(m'_{vu})max(m'_{vu}m'_{uw},0)=m_{vw}+sgn(m_{vu})max(m_{vu}m_{uw},0)$$$$+sgn(-m_{vu})max((-m_{vu})(-m_{uw}),0)=m_{vw}.$$ Thus $M''=M$ and the mutation is an involutive.

(ii) Suppose now that $u,w$ are not linked by an arrow in the quiver $Q$. Then, take the previous notation $M,M'$; moreover, define the matrices $M''$, $P'$ and $P''$: $M''$ is obtained from $M'$ by a mutation at $w$, $P'$ is obtained from $M$ by a mutation at $w$ and $P''$ is obtained from $P'$ by a mutation at $u$. We verify that $P''=M''$, which will prove the commutation.

Since $u,w$ are unconnected in $Q$, we have $m_{uw}=0=m_{wu}$, thus also $m'_{uw}=0=m'_{wu}$. The same relations holds for $M'',P'$ and $P''$.

We show first that $m''_{uv}=p''_{uv}$. If $v=w$, these numbers are 0, so we may assume that $v\neq w$. We have (mutation at $u$): $m'_{uv}=-m_{uv}$ and $p''_{uv}=-p'_{uv}$. Moreover, since $u\neq w$ and $v \neq w$, mutating at $w$, we have $m''_{uv}=m'_{uv}+sgn(m'_{uw})max(m'_{uw}m'_{wv},0)=m'_{uv}=-m_{uv}$. Similarly, $p'_{uv}=m_{uv}+sgn(m_{uw})max(m_{uw}m_{wv},0)=m_{uv}$. Thus $m''_{uv}=-m_{uv}=-p'_{uv}=p''_{uv}$.

The proof that $m''_{vu}=p''_{vu}$, $m''_{wv}=p''_{wv}$ and $m''_{vw}=p''_{vw}$ are similar.

We assume now that $v,t\neq u,w$ and show that $m''_{tv}=p''_{tv}$. We have (mutation at $u$): $m'_{tv}=m_{tv}+sgn(m_{tu})max(m_{tu}m_{uv},0)$ and $p''_{tv}=p'_{tv}+sgn(p'_{tu})max(p'_{tu}p'_{uv},0)$. Moreover, we have (mutation at $w$): $m''_{tv}=m'_{tv}+sgn(m'_{tw})max(m'_{tw}m'_{wv},0)$ and $p'_{tv}=m_{tv}+sgn(m_{tw})max(m_{tw}m_{wv},0)$. We have also (mutation at $w$): $p'_{tu}=m_{tu}+sgn(m_{tw})max(m_{tw}m_{wu},0)=m_{tu}$ since $m_{wu}=0$. Similarly, $p'_{uv}=m_{uv}$, $m'_{tw}=m_{tw}$ and $m'_{wv}=m_{wv}$. Thus $m''_{tv}=m_{tv}+sgn(m_{tu})max(m_{tu}m_{uv},0)+sgn(m_{tw})max(m_{tw}m_{wv},0)$ and $p''_{tv}=m_{tv}+sgn(m_{tw})max(m_{tw}m_{wv},0)+sgn(m_{tu})max(m_{tu}m_{uv},0)$. Thus $m''_{tv}=p''_{tv}$.

(iii) We prove it first in the case where there is no arrow between $u$ and $w$ in $Q$. Then $A$ is independent of $w$.
Moreover $C$ is the mutation polynomial of $\mu_w\mu_u(Q)$ at $u$, so that, by a previous remark, it is equal to that of $\mu_u\mu_w\mu_u(Q)$ at $u$. But by (ii), this is the mutation polynomial of $\mu_w(Q)$ at $u$; since the mutation at $w$ does not change the arrows incident to $u$ (because $u$ and $w$ are not incident in $Q$), this polynomial is also equal to the mutation polynomial of $Q$ at $u$. Thus we have $C=A$ and (iii) holds in this case.

Suppose now that there is an arrow between $u$ and $w$ in $Q$. We may by symmetry assume that there is an arrow $u\rightarrow w$ with multiplicity $a$ in $\mu_u(Q)$. Denote by $M$ the antisymmetric matrix associated to this quiver. In the next formulas, $v$ is always assumed to be $\neq u,w$. We have (since $A$ is also the mutation polynomial of $\mu_u(Q)$ at $u$)
$$A=\prod_v v^{m_{vu}}+w^a\prod_v v^{m_{uv}}.$$
Now, since $m_{uw}=a>0$, $B$ depends on $u$ and therefore $B_0$ is equal to 
$B_0=\prod_v v^{m_{wv}}$.
Now, by definition of the mutation of $\mu_u(Q)$ at $w$, we see that $C$ is equal to $C'$ divided by some monomial, with
$$C'=w^a\prod_v v^{m_{vu}}+\prod_v v^{m_{uw}+am_{wv}}.$$
Indeed, this corresponds to apply the first two steps of the mutation at $w$, without the removal of 2-cycles; then the removal of 2-cycles amounts to divide the mutation polynomial by some monomial.
Now 
$$A\mid_{w\leftarrow B_0/w}=\prod_v v^{m_{vu}}+\frac{\prod_v v^{am_{wv}}}{w^a}\prod_v v^{m_{uw}}.$$
Thus $C'=w^a\times A\mid_{w\leftarrow B_0/w}$, which ends the proof.
\end{proof}

\begin{theorem}\label{LaurentQuiver}
Each element of the basis of each quiver obtained from $Q_0$ by a sequence of mutations is a Laurent polynomials in the variables $v\in V$.
\end{theorem}

This theorem will be proved in the next subsection.

\begin{corollary}
Each frieze is integer-valued.
\end{corollary}

\begin{proof}
Replace each initial value $a(v,0)=1$ by the variable $v$. It is enough to show that the sequences $a(v,n)$ are Laurent polynomials, since one may recover the frieze by specializing these variables to 1. Now, one verifies that computing the sequences using the recursion (\ref{friserecursion}) amounts to mutate only of vertices which are sources of the quiver. Hence, the corollary is a particular case of Theorem \ref{LaurentQuiver}.
\end{proof}

\subsection{The Laurent phenomenon for quiver mutation}

Consider the {\em mutation graph} of $Q_0$: its vertices are the quivers $Q$ obtained by a sequence of mutations from $Q_0$, 
with an edge $(Q,Q')$ labelled $v$ if there is mutation at $v$ between $Q$ and $Q'$. We prove the theorem by induction on the distance from $Q_0$ to $Q$ in this graph. We start by considering a special case of distance 3. 

We shall need an easy lemma.

\begin{lemma} \label{relprime}
Let $R$ be a factorial ring and $u$ a variable. Let $au+b$ be a polynomial of degree $1$ in $R[u]$ and $P\in R$ be such that $a$ and $P$ are relatively prime in $R$. Then $au+b$ and $P$ are relatively prime in $R[u^{\pm 1}]$.
\end{lemma}

\begin{proof}
Let $p\in R[u^{\pm 1}]$ be an irreducible divisor of $au+b$ and $P$. Since $u$ is invertible in $R[u^{\pm 1}]$, we may assume that $p$ is a polynomial in $R[u]$ with nonzero constant term. Since $p$ divides $au+b$ in $R[u^{\pm 1}]$, we see that $p$ is of degree 0 or 1. If $p$ is of degree 0, then $p$ divides $a$ and $b$ in $R[u]$, and also $P$,a contradiction. Now, $p$ cannot be of degree 1, since it divides $P$, which is in $R$.
\end{proof}

\begin{proposition}\label{braid}
Consider $4$ quivers $Q_0,Q_1,Q_2,Q_3$, which are linked by mutations $Q_0\relbar Q_1 \relbar Q_2 \relbar Q_3$, 
with corresponding bases $V,X,Y,Z$, with respective mutations at $u,w,u$, 
and with respective mutations polynomials $A,B,C$. We assume that $A$ and $B_0=B\mid_{u\leftarrow 0}$ are relatively prime in $L$. Then the elements of these bases are Laurent polynomials and one has: $gcd(x_u,y_w)=1=gcd(x_u,z_u)$ in $L$.
\end{proposition}

\begin{proof}
If $v$ is distinct from $u,w$, then the mutations of the proposition do not change the corresponding element, so that $v=x_v=y_v=z_v$. For a similar reason, since $u\neq w$, we have, $w=x_w$, $x_u=y_u$ and $y_w=z_w$. 

Before pursuing, we introduce a slight abuse of notation. We know that the polynomials $A$ and $C$ do not depend on $u$ and that $B$ does not depend on $w$. We therefore write $A(w)$ for $A$. Similarly we write $B(u)$ and $C(w)$. We also use the notation $B^x=B(x_v,v\in V)$.

We have $x_u=A/u$, so that $x_u$ is a Laurent polynomial. Moreover, $y_w=B^x/x_w=B(x_u)/w$, since $w=x_w$ and since the only variable changed in the first mutation is $x_u$; thus $y_w=B(A/u)/w$, so that $y_w$ is a Laurent polynomial, too.  

Thus, regarding Laurentness in the proposition, it remains only to show that $z_u$ is a Laurent polynomial.
We have $z_u=C^y/y_u=C(y_w)/y_u$. Thus $z_u=C(B(x_u)/w))/y_u$. Denote $B_0=B(0)$. Then, since $y_u=x_u$, $z_u=(C(B(x_u)/w)-C(B_0/w))/x_u + C(B_0/w)/x_u$. 

Now, we have $B(x_u)/w\equiv B_0/w \, \, mod. \,x_u$, since $B(x_u)$ is a polynomial in $x_u$ with constant term $B_0$. Thus $(C(B(x_u)/w)-C(B_0/w))/x_u$ is polynomial in $x_u$ and therefore a Laurent polynomial. Moreover, $C\mid_{w\leftarrow B_0/w}=mA$, for some Laurent monomial $m$, by Proposition \ref{mutprop} (iii), so that $C(B_0/w)/x_u=mA/x_u=mu$, which is a Laurent monomial. Thus, $z_u$ is a Laurent polynomial.

By the above congruence, we have $y_w \equiv B_0/w \, mod. \, x_u$. Thus, since $u$ and $w$ are invertible, $gcd(x_u,y_w)=gcd(A,B_0)=1$ in $L$, since $A,B_0$ are relatively prime by hypothesis.

Let $f(u)=C(B(u)/w)$. We have seen that $z_u=(f(x_u)-f(0))/x_u+mu$. Modulo $x_u$, we have $(f(x_u)-f(0))/x_u \equiv f'(0)=C'(B_0/w)B'(0)/w$. Thus $z_u \equiv C'(B_0/w)B'(0)/w+mu$. Note that $C'(B_0/w)B'(0)/w$ and $m$ do not depend on $u$. Now, we have $gcd(A,m)=1$, since $m$ is a Laurent monomial. Thus $gcd(A,C'(B_0/w)B'(0)/w+mu)=1$, as follows from the lemma, applied to $R$ equal to the ring of Laurent polynomials over $\mathbb Z$ in the variables different from $u$, using the fact that $A$ is independent of $u$. Thus $gcd(x_u,z_u)=1$ in $L$.
\end{proof}

\begin{proof} (of Theorem \ref{LaurentQuiver})
Let $Q$ be some quiver in the mutation class of $Q_0$, with associated basis $X$. If the distance of $Q$ to $Q_0$ in the mutation graph is at most $2$, then the result is staightforward. Suppose now that the distance is  at least three and denote respectively by $u,w$ the vertices at which are performed the two first mutations. We may assume that $u\neq w$.

Let $Q_0,Q_1,\ldots,Q_n=Q$, $n\leq 3$, 
be the successive quivers on a shortest path from $Q_0$ to $Q$ in the mutation graph, and denote by $A$, $B$ the mutations polynomials of $Q_0$ at $u$ and of $Q_1$ at $w$. 

We show first that we may assume that $A$ and $B_0$ are relatively prime. This is clear if $u,w$ are linked by an arrow in $Q_0$ since $B_0$ is then a monomial by Proposition \ref{mutprop} (iii). Suppose now that there is no
arrow in $Q_0$ between $u$ and $w$; then we add to $Q_0$ a new vertex $h$, with a arrow $h\rightarrow u$. If we apply to this new quiver $Q'_0$ the same sequence of mutations as the one applied to $Q_0$ in order to obtain $Q$, then we never mutate at $h$; call $Q'$ the last quiver of this new sequence; if we show that the 
basis elements of $Q'$ are Laurent polynomial in the variables $V\cup \{h\}$, then the basis elements of $Q$ will be Laurent polynomials in $V$, since the latter are obtained by putting $h=1$ in the former. Now, the new mutation polynomial $A'$ is of degree 1 in $h$, and $B'=B$ is independent of $h$, since there is no arrow between $h$ and $w$ in $Q'_1$ (because this is true for $Q'_0$ and since there is no arrow between $u$ and $w$ in $Q'_0$). Thus $A'$ and $B'_0$ are relatively prime in $L$ by Lemma \ref{relprime}; indeed, the coefficient of $h$ in $A'$ is a monomial in the variables in $V$, and no variable in $V$ divides both monomials whose sum is $B'=B'_0$.

Let $X,Y$ be the bases attached to $Q_1,Q_2$. Let $Q'_3$ be the quiver obtained by mutation at $u$ from $Q_2$, with its basis denoted $Z$. Then the quivers $Q_0,Q_1,Q_2,Q'_3$ satisfy the hypothesis of Proposition \ref{braid}. By induction, since the distances in the mutation graph from $Q_1$ to $Q$ and from $Q'_3$ to $Q$ are shorter than the distance from $Q_0$ to $Q$, we have that each element $p$ of the basis attached to $Q$ belongs to $\mathbb Z[x_v^{\pm 1}]$ and also to $\mathbb Z[z_v^{\pm 1}]$. This implies, since $x_v=v$ for any $v\neq u$, that in $p\in L/x_u^i$ for some natural number $i$. Similarly, $p$ is in $L/z_u^jz_w^k$. We have $z_w=y_w$ and $gcd (x_u^i,z_u^jz_w^k)=1$ by Proposition \ref{braid}. Thus $p$ is in $L$.

\end{proof}

\section{Dynkin diagrams}\label{Dd}

These diagrams are well-known in many classification theorems, including the first one (historically): the classification of simple Lie algebras by Cartan and Killing. 

The Dynkin diagrams and the extended Dynkin diagrams  are shown in Figure \ref{dynkin} and Figure \ref{extended}. Note that the index of a Dynkin diagram is its number of nodes, whereas the number of nodes of an extended Dynkin diagram is one more than its index.

\newsavebox{\An}
\savebox{\An}[8cm][b]{
\begin{picture}(8,1)(0,0)
\multiput(0,0)(1,0){3}{\cercle}
\multiput(0.05,0)(1,0){2}{\segment}
\put(2,-0.02){\makebox[2cm]{$\ldots$}}
\multiput(4,0)(1,0){2}{\cercle}
\multiput(4.05,0)(1,0){1}{\segment}
\end{picture}}

\newsavebox{\Dn}
\savebox{\Dn}[8cm][b]{
\begin{picture}(8,1)(-0.8321,0)
\multiput(0,0)(1,0){3}{\cercle}
\multiput(0.05,0)(1,0){2}{\segment}
\put(2,-0.02){\makebox[2cm]{$\ldots$}}
\multiput(4,0)(1,0){2}{\cercle}
\multiput(4.05,0)(1,0){1}{\segment}
\put(-0.8321,0.5547){\cercle}
\put(-0.8321,-0.5547){\cercle}
\put(-0.7904,0.5270){\line(3,-2){0.75}}
\put(-0.7904,-0.5270){\line(3,2){0.75}}
\end{picture}}

\newsavebox{\Esix}
\savebox{\Esix}[8cm][b]{
\begin{picture}(8,1)(0,0)
\multiput(0,0)(1,0){5}{\cercle}
\multiput(0.05,0)(1,0){4}{\segment}
\put(2,0.05){\line(0,1){0.9}}
\put(2,1){\cercle}
\end{picture}}

\newsavebox{\Esept}
\savebox{\Esept}[8cm][b]{
\begin{picture}(8,1)(0,0)
\multiput(0,0)(1,0){6}{\cercle}
\multiput(0.05,0)(1,0){5}{\segment}
\put(2,0.05){\line(0,1){0.9}}
\put(2,1){\cercle}
\end{picture}}

\newsavebox{\Ehuit}
\savebox{\Ehuit}[8cm][b]{
\begin{picture}(8,1)(0,0)
\multiput(0,0)(1,0){7}{\cercle}
\multiput(0.05,0)(1,0){6}{\segment}
\put(2,0.05){\line(0,1){0.9}}
\put(2,1){\cercle}
\end{picture}}

\begin{figure}[ht]
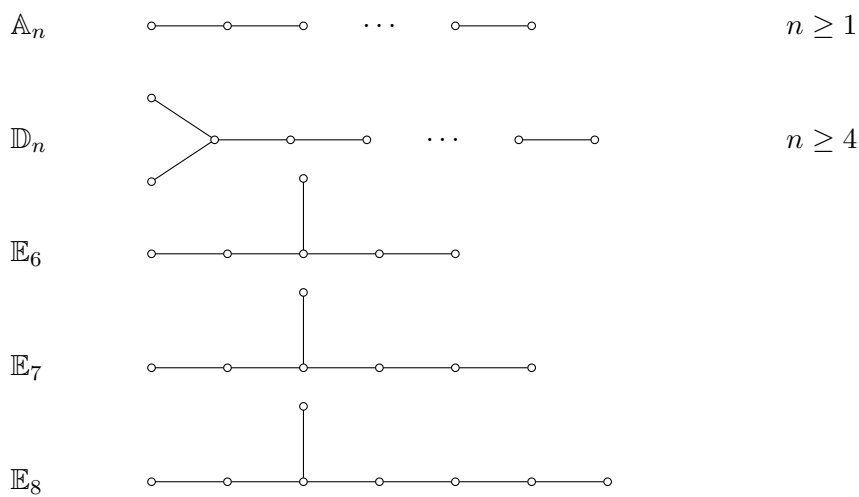

\[
\renewcommand{\arraystretch}{3}
\begin{array}{lll}
\makebox[1.5cm][l]{$\mathbb{A}_n$} & \raisebox{0.1cm}{\usebox{\An}} & n \ge 1 \\
\mathbb{D}_n & \raisebox{0.1cm}{\usebox{\Dn}} & n \ge 4 \\
\mathbb{E}_6 & \raisebox{0.1cm}{\usebox{\Esix}} \\
\mathbb{E}_7 & \raisebox{0.1cm}{\usebox{\Esept}} \\
\mathbb{E}_8 & \raisebox{0.1cm}{\usebox{\Ehuit}}
\end{array}
\]
\caption{The Dynkin Diagrams}\label{dynkin}
\end{figure}

\newsavebox{\Antilde}
\savebox{\Antilde}[2.4cm][b]{
\setlength{\unitlength}{2cm}
\begin{picture}(2.4,2.4)(-1.2,-1.2)
\put(-0.7660, -0.6428){\circle{0.05}}
\put(-1,0){\circle{0.05}}
\put(-0.7660, 0.6428){\circle{0.05}}
\put(-0.1736, 0.9848){\circle{0.05}}
\put(0.5, 0,8660){\circle{0.05}}
\put(0.5, 0,8660){\circle{0.05}}

\qbezier(-0.7787, -0.6273)(-0.9866, -0.3591)(-0.9998, -0.02000)
\qbezier(-0.9998, 0.02000)(-0.9986, 0.3244)(-0.7787, 0.6273)
\qbezier(-0.7530, 0.6579)(-0.38, 0.97)(-0.1884, 0.981)
\qbezier(-0.1589, 0.9873)(0.15, 1.06)(0.4826, 0.8759)

\put(0.939692621, -0.342020143){\circle{0.01}}
\put(0.835487811, -0.549508978){\circle{0.01}}
\put(0.686241638,	-0.727373642){\circle{0.01}}
\put(0.5, -0.866025404){\circle{0.01}}

\put(0.686241638, 0.727373642){\circle{0.01}}
\put(0.835487811, 0.549508978){\circle{0.01}}
\put(0.939692621, 0.342020143){\circle{0.01}}
\put(0.993238358, 0.116092914){\circle{0.01}}
\put(0.993238358, -0.116092914){\circle{0.01}}

\put(0.286803233, -0.957989512){\circle{0.01}}
\put(0.058144829, -0.998308158){\circle{0.01}}
\put(-0.173648178, -0.984807753){\circle{0.01}}
\put(-0.396079766, -0.918216107){\circle{0.01}}
\put(-0.597158592, -0.802123193){\circle{0.01}}
\end{picture}}

\newsavebox{\Dntilde}
\savebox{\Dntilde}[8cm][b]{
\begin{picture}(8,1)(-0.8321,0)
\multiput(0,0)(1,0){3}{\cercle}
\multiput(0.05,0)(1,0){2}{\segment}
\put(2,-0.02){\makebox[2cm]{$\ldots$}}
\multiput(4,0)(1,0){2}{\cercle}
\multiput(4.05,0)(1,0){1}{\segment}
\put(-0.8321,0.5547){\cercle}
\put(-0.8321,-0.5547){\cercle}
\put(-0.7904,0.5270){\line(3,-2){0.75}}
\put(-0.7904,-0.5270){\line(3,2){0.75}}
\put(5.8321,0.5547){\cercle}
\put(5.8321,-0.5547){\cercle}
\put(5.7904,0.5270){\line(-3,-2){0.75}}
\put(5.7904,-0.5270){\line(-3,2){0.75}}
\end{picture}}

\newsavebox{\Esixtilde}
\savebox{\Esixtilde}[8cm][b]{
\begin{picture}(8,1)(0,0)
\multiput(0,0)(1,0){5}{\cercle}
\multiput(0.05,0)(1,0){4}{\segment}
\multiput(2,0.05)(0,1){2}{\segmentvertical}
\multiput(2,1)(0,1){2}{\cercle}
\end{picture}}

\newsavebox{\Esepttilde}
\savebox{\Esepttilde}[8cm][b]{
\begin{picture}(8,1)(0,0)
\multiput(0,0)(1,0){7}{\cercle}
\multiput(0.05,0)(1,0){6}{\segment}
\put(3,0.05){\line(0,1){0.9}}
\put(3,1){\cercle}
\end{picture}}

\newsavebox{\Ehuittilde}
\savebox{\Ehuittilde}[8cm][b]{
\begin{picture}(8,1)(0,0)
\multiput(0,0)(1,0){8}{\cercle}
\multiput(0.05,0)(1,0){7}{\segment}
\put(2,0.05){\line(0,1){0.9}}
\put(2,1){\cercle}
\end{picture}}

\begin{figure}[ht]
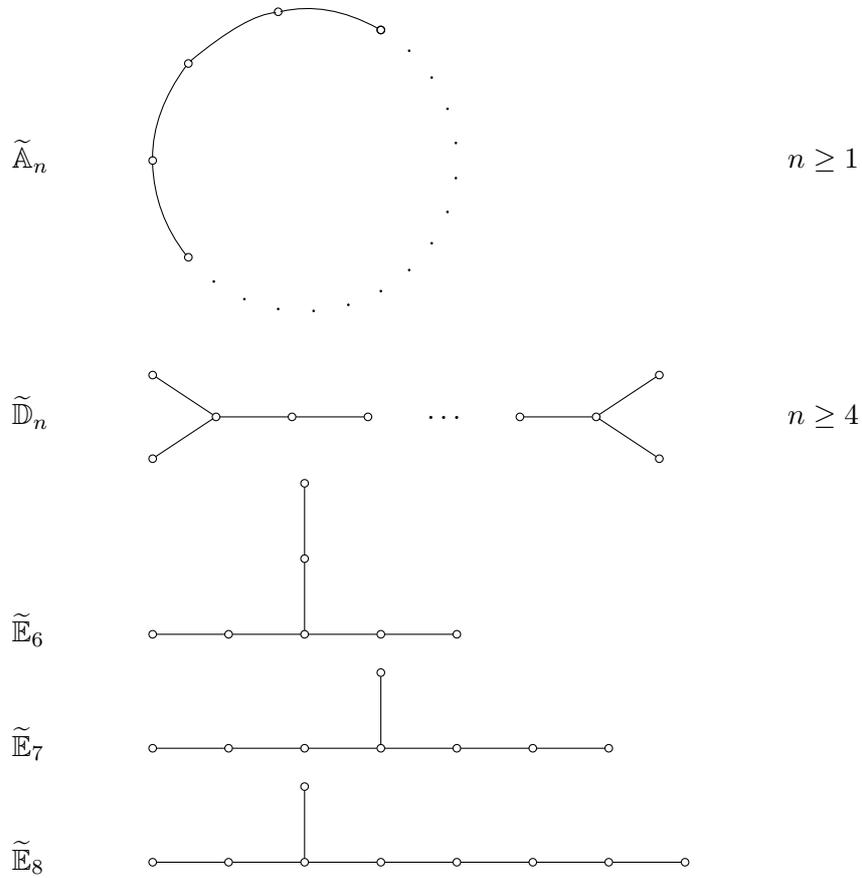

\[
\renewcommand{\arraystretch}{3}
\begin{array}{lll}
\makebox[1.5cm][l]{$\widetilde{\mathbb{A}}_n$} & \hspace*{2cm} \raisebox{-2.3cm}{\usebox{\Antilde}} & n \ge 1 \\
\widetilde{\mathbb{D}}_n & \raisebox{0.1cm}{\usebox{\Dntilde}} & n \ge 4 \\\\
\widetilde{\mathbb{E}}_6 & \raisebox{0.1cm}{\usebox{\Esixtilde}} \\
\widetilde{\mathbb{E}}_7 & \raisebox{0.1cm}{\usebox{\Esepttilde}} \\
\widetilde{\mathbb{E}}_8 & \raisebox{0.1cm}{\usebox{\Ehuittilde}}
\end{array}
\]
\caption{The extended Dynkin Diagrams}\label{extended}
\end{figure}

They have a very simple combinatorial characterization, due to Vinberg (for Dynkin diagrams) \cite{V} and Berman-Moody-Wonenburger (for extended Dynkin diagrams) \cite{BMW}. 

Let $G$ be a graph without loops. An {\em additive} (resp. {\em subadditive}) function is a function from the set $V$ of vertices of $G$ into the {\em positive real numbers} such that for any vertex $v$, $2f(v)$ is equal (resp. is greater or equal ) to$\sum_{\{v,w\}\in E}f(w)$.

\begin{theorem}\label{Dynkin}
A simple and connected graph $G$ is a Dynkin diagram (resp. an extended Dynkin diagram) if and only if it has a subbadditive function which is not additive (resp. it has an additive function).
\end{theorem}

The additive functions for extended Dynkin diagram are shown in Figure \ref{functions}. For the proof of the theorem, we shall follow \cite{HPR} (the "if" part of their proof assumes that the function is integer-valued, but it extends without change to real-valued functions).

\begin{proposition}\label{containsDynki}
Let $G$ be a finite connected graph. Then either $G$ is a Dynkin diagram or it contains an extended Dynkin diagram as subgraph.
\end{proposition}

We say that $G'$ is a subgraph of $G$ if the set of vertices of $G'$ is contained in that of $G$, and likewise for the set of edges.

\newcommand{\se}[1]{$\underbrace{\sllq}_{#1}$}

\newsavebox{\accf}
\savebox{\accf}[0cm]{$\left. \rule{0cm}{1.9cm} \hspace{0.3cm} \right\}$
\begin{picture}(0,3.5)(0,1.75) \put(0,0){\makebox(0,3.5){\raisebox{0.2cm}{$\scriptstyle r$}}} \end{picture}}

\begin{figure}[ht]
\begin{center}
\makebox[3.6cm]{$
\underbrace{
\begin{picture}(3.6,0.6)(-0.05,-0.3)
\put(0,0){\cercle}
\put(0.05,0){\line(1,0){0.9}}
\put(1.05,0){\makebox[1.6cm]{$\ldots$}}
\put(2.55,0){\line(1,0){0.9}}
\put(3.5,0){\cercle}
\end{picture}}_{s}$}
\unskip
\makebox[0.9cm][b]{\begin{picture}(0.9,0.6)(0,-0.3) \put(0,0){\line(1,0){0.92}} \end{picture}}
\unskip
\hspace*{-0.35cm}
\makebox[0.6cm]
{\begin{picture}(0,3.6)(0,-0.3)
\put(0,0){\cercle}
\put(0,0.05){\line(0,1){0.9}}
\put(0,1){\makebox(0,1.5){\raisebox{0.2cm}{$\vdots$}}}
\put(0,2.55){\line(0,1){0.9}}
\put(0,3.5){\cercle}
\end{picture}} \unskip
\raisebox{1.97cm}{\usebox{\accf}} \unskip
\unskip
\hspace*{-0.35cm}
\makebox[0.9cm][b]{\begin{picture}(0.9,0.6)(0,-0.3) \put(-0.02,0){\line(1,0){0.92}} \end{picture}}
\unskip
\makebox[3.6cm]{$
\underbrace{
\begin{picture}(3.6,0.6)(-0.05,-0.3)
\put(0,0){\cercle}
\put(0.05,0){\line(1,0){0.9}}
\put(1.05,0){\makebox[1.6cm]{$\ldots$}}
\put(2.55,0){\line(1,0){0.9}}
\put(3.5,0){\cercle}
\end{picture}}_{t}$}
\end{center}
\caption{}\label{graph}
\end{figure}

\begin{proof}
We show that if $G$ does not contain any extended Dynkin diagram, then it is a Dynkin diagram. The fact that $G$ does not contain $\tilde {\mathbb A}_n$ implies that $G$ is acyclic, hence is a tree (because it is connected). Since $G$ does not contain $\tilde {\mathbb D}_4$, no vertex has more than 3 neighbours. Since $G$ does not contain $\tilde {\mathbb D}_n$ for $n\geq 5$, at most one vertex in $G$ has 3 neighbours. Thus $G$ is of the form shown in Figure \ref{graph},
with $r\leq s\leq t$. Since $G$ does not contain $\tilde {\mathbb E}_6$, we must have $r\leq 1$. If $r=0$, then $G={\mathbb A}_{n}$. Suppose now that $r=1$, thus also $1\leq s$. Since $G$ does not contain $\tilde {\mathbb E}_7$, we must have $1\leq s\leq 2$. If $s=1$, then $G= {\mathbb D}_n$. Suppose now that $s=2$, hence also $2\leq t$. Since $G$ does not contain $\tilde {\mathbb E}_8$, we must have $2\leq t\leq 4$; in these 3 cases, we have $G={\mathbb E}_6$, ${\mathbb E}_7$ or ${\mathbb E}_8$.
\end{proof}

\begin{proposition}\label{additive}
Le $G$ be an extended Dynkin diagram and $f$ be a subadditive function for $G$. Then $f$ is additive.
\end{proposition}

\begin{proof}
Let $C$ be the Cartan matrix of $G$, that is the $V\times V$ matrix with 2's on the diagonal, a $-1$ at entry $(v,w)$ if $\{v,w\}$ is an edge and 0 elsewhere. Let also $F$ denote the row vector $(f(v))_{v\in V}$. Then the fact that $f$ is subadditive means that $FC\leq 0$ componentwise. Let $h$ be an additive function for $G$: it exists, see Figure \ref{functions}. Since $C$ is symmetric, we have by additivity of $h$, $CH=0$, where $H$ is the row vector $(h(v))_{v\in V}$. Thus we have $FCH=0$. Now, the components of $H$ are positive and those of $FC$ are $\leq 0$. Thus we must have $FC=0$ and $f$ is an additive function.
\end{proof}

\begin{proposition}\label{notadditive}
Suppose that $G'$ is a subgraph of $G$. Then the restriction to $G'$ of any subadditive function $f$ of $G$ is a subbadditive function of $G$. Moreover, if $G'$ is a proper subgraph, then the restriction is not additive .
\end{proposition}

\begin{proof}
Let $v$ be a vertex of $G'$. Then $$2f(v)\geq \sum_{\{v,w\}\in E}f(w) \geq \sum_{\{v,w\}\in E'}f(w),$$ so that $f$ is a subadditive function for $G'$. 

Suppose now that $G'$ is a proper subgraph. Arguing inductively, we may assume that $G'$ has the same vertices as $G$ and one edge less, or that $G'$ has one vertex less.
In the first case, let  $\{v,u\}$ be this edge; then, since the values of $f$ are positive, $2f(v)>\sum_{\{v,w\}\in E'}f(w)$ and $f\mid G'$ is not additive.
In the second case, let $u$ be this vertex; $u$ is not isolated in $G$, since $G$ is connected; let $v$ be a neighbour of $u$ in $G$; then $v$ is in $G'$, and $2f(v)>\sum_{\{v,w\}\in E'}f(w)$, so that $f$ is not additive on $G'$.
\end{proof}

\setlength{\unitlength}{1cm}

\newsavebox{\Antildeet}
\savebox{\Antildeet}[2.4cm][b]{
\setlength{\unitlength}{2cm}
\begin{tikzpicture} [>=latex, scale=2, fill opacity = 1]
\draw (0.5,0.866) arc (60:220:1); 
\node at (-0.7660, -0.6428) [fill = white] {$1$};
\node at (-1,0) [fill = white] {$1$};
\node at (-0.7660, 0.6428) [fill = white] {$1$};
\node at (-0.1736, 0.9848) [fill = white] {$1$};
\node at (0.5, 0.8660) [fill = white] {$1$};
\put(0.939692621, -0.342020143){\circle{0.01}}
\put(0.835487811, -0.549508978){\circle{0.01}}
\put(0.686241638,	-0.727373642){\circle{0.01}}
\put(0.5, -0.866025404){\circle{0.01}}
\put(0.686241638, 0.727373642){\circle{0.01}}
\put(0.835487811, 0.549508978){\circle{0.01}}
\put(0.939692621, 0.342020143){\circle{0.01}}
\put(0.993238358, 0.116092914){\circle{0.01}}
\put(0.993238358, -0.116092914){\circle{0.01}}
\put(0.286803233, -0.957989512){\circle{0.01}}
\put(0.058144829, -0.998308158){\circle{0.01}}
\put(-0.173648178, -0.984807753){\circle{0.01}}
\put(-0.396079766, -0.918216107){\circle{0.01}}
\put(-0.597158592, -0.802123193){\circle{0.01}}
\end{tikzpicture}}

\newsavebox{\Dntildeet}
\savebox{\Dntildeet}[8cm][b]{
\begin{tikzpicture} [>=latex]
\multiput(0,0)(1,0){3}{\cercle}
\multiput(0.05,0)(1,0){2}{\segment}
\put(2,-0.02){\makebox[2cm]{$\ldots$}}
\multiput(4,0)(1,0){2}{\cercle}
\multiput(4.05,0)(1,0){1}{\segment}
\put(-0.8321,0.5547){\cercle}
\put(-0.8321,-0.5547){\cercle}
\put(-0.7904,0.5270){\line(3,-2){0.75}}
\put(-0.7904,-0.5270){\line(3,2){0.75}}
\put(5.8321,0.5547){\cercle}
\put(5.8321,-0.5547){\cercle}
\put(5.7904,0.5270){\line(-3,-2){0.75}}
\put(5.7904,-0.5270){\line(-3,2){0.75}}
\node at (-0.8321,0.5547) [fill = white] {$1$};
\node at (-0.8321,-0.5547) [fill = white] {$1$};
\node at (0,0) [fill = white] {$2$};
\node at (1,0) [fill = white] {$2$};
\node at (2,0) [fill = white] {$2$};
\node at (4,0) [fill = white] {$2$};
\node at (5,0) [fill = white] {$2$};
\node at (5.8321,0.5547) [fill = white] {$1$};
\node at (5.8321,-0.5547) [fill = white] {$1$};
\end{tikzpicture}}

\newsavebox{\Esixtildeet}
\savebox{\Esixtildeet}[8cm][b]{
\begin{tikzpicture} [>=latex]
\multiput(0,0)(1,0){5}{\cercle}
\multiput(0.05,0)(1,0){4}{\segment}
\multiput(2,0.05)(0,1){2}{\segmentvertical}
\multiput(2,1)(0,1){2}{\cercle}
\node at (0,0) [fill = white] {$1$};
\node at (1,0) [fill = white] {$2$};
\node at (2,0) [fill = white] {$3$};
\node at (3,0) [fill = white] {$2$};
\node at (4,0) [fill = white] {$1$};
\node at (2,1) [fill = white] {$2$};
\node at (2,2) [fill = white] {$1$};
\end{tikzpicture}}

\newsavebox{\Esepttildeet}
\savebox{\Esepttildeet}[8cm][b]{
\begin{tikzpicture} [>=latex]
\multiput(0,0)(1,0){7}{\cercle}
\multiput(0.05,0)(1,0){6}{\segment}
\put(3,0.05){\line(0,1){0.9}}
\put(3,1){\cercle}
\node at (0,0) [fill = white] {$1$};
\node at (1,0) [fill = white] {$2$};
\node at (2,0) [fill = white] {$3$};
\node at (3,0) [fill = white] {$4$};
\node at (4,0) [fill = white] {$3$};
\node at (5,0) [fill = white] {$2$};
\node at (6,0) [fill = white] {$1$};
\node at (3,1) [fill = white] {$2$};
\end{tikzpicture}}

\newsavebox{\Ehuittildeet}
\savebox{\Ehuittildeet}[8cm][b]{
\begin{tikzpicture} [>=latex]
\multiput(0.05,0)(1,0){7}{\segment}
\put(2,0.05){\line(0,1){0.9}}
\node at (0,0) [fill = white] {$2$};
\node at (1,0) [fill = white] {$4$};
\node at (2,0) [fill = white] {$6$};
\node at (3,0) [fill = white] {$5$};
\node at (4,0) [fill = white] {$4$};
\node at (5,0) [fill = white] {$3$};
\node at (6,0) [fill = white] {$2$};
\node at (7,0) [fill = white] {$1$};
\node at (2,1) [fill = white] {$3$};
\end{tikzpicture}}

\begin{figure}[ht]
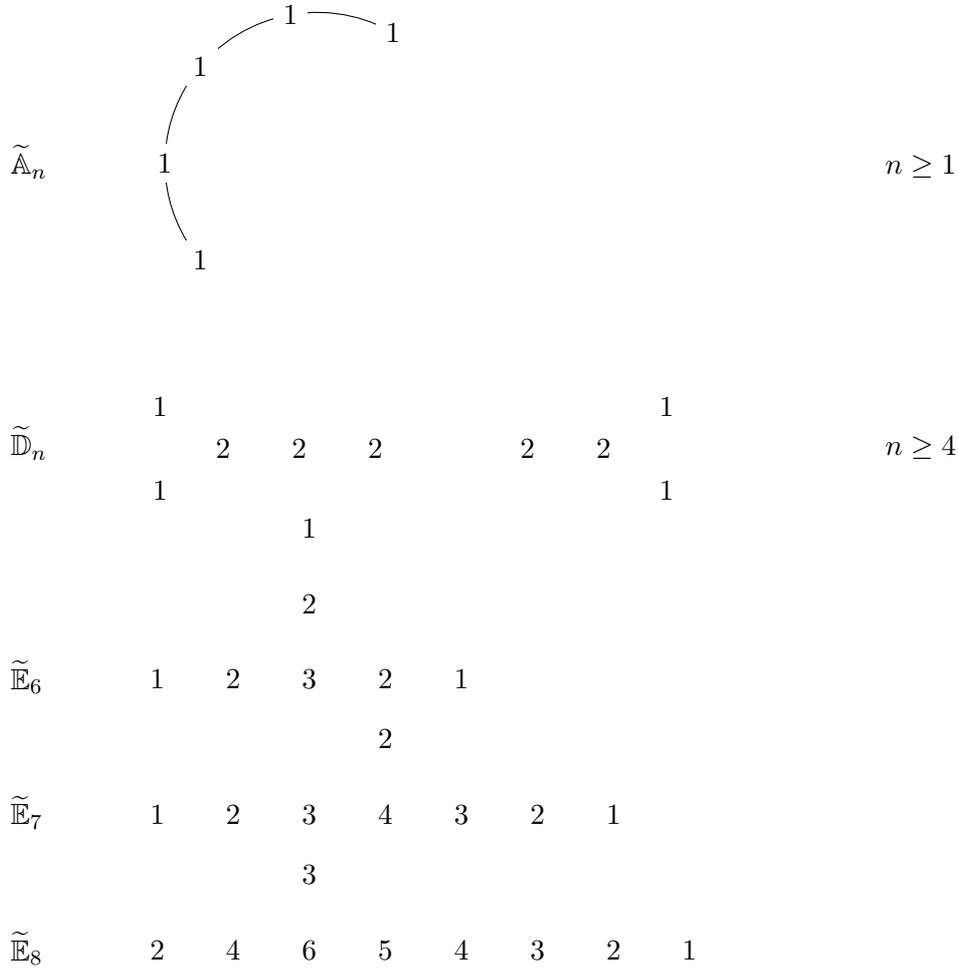

\[
\renewcommand{\arraystretch}{3}
\begin{array}{lll}
\makebox[1.5cm][l]{$\widetilde{\mathbb{A}}_n$} & \hspace*{1cm} \raisebox{-1.4cm}{\usebox{\Antildeet}} & \hspace{2cm}n \ge 1 \\\\
\widetilde{\mathbb{D}}_n & \hspace{-1cm}\raisebox{-0.72cm}{\usebox{\Dntildeet}} & \hspace{2cm}n \ge 4 \\
\widetilde{\mathbb{E}}_6 & \hspace{-3.7cm}\raisebox{-0.12cm}{\usebox{\Esixtildeet}} \\
\widetilde{\mathbb{E}}_7 & \hspace{-1.7cm}\raisebox{-0.12cm}{\usebox{\Esepttildeet}} \\
\widetilde{\mathbb{E}}_8 & \hspace{-0.7cm}\raisebox{-0.12cm}{\usebox{\Ehuittildeet}}
\end{array}
\]
\caption{The additive functions of extended Dynkin Diagrams}\label{functions}
\end{figure}

\begin{proof} (of the theorem)
If $G$ is an extended Dynkin diagram, then $G$ has an additive function: see Figure \ref{functions}. If $G$ is a Dynkin diagram, then it is a proper subgraph of some extended Dynkin diagram, so that it has a subadditive function, which is not additive by Proposition \ref{notadditive}.

Conversely suppose that $G$ has an additive function $f$. If $G$ is not a Dynkin diagram, then by Proposition \ref{containsDynkin}, it contains an extended Dynkin diagram $G'$. If $G'$ is a proper subgraph, then $f\mid G'$ is not additive by Proposition \ref{notadditive}, which contradicts Proposition \ref{additive}. Thus, $G=G'$ and $G$ is an extended Dynkin diagram.

Suppose now that $G$ has a subadditive function $f$ which is not additive. If $G$ is not a Dynkin diagram, then by Proposition \ref{containsDynki}, it contains an extended Dynkin diagram $G'$; then $f\mid G'$ is subadditive but not additive by Proposition \ref{notadditive} (indeed, either $G'=G$, or $G'$ is a proper subgraph of $G$); this contradicts Proposition \ref{additive}.
\end{proof}

\section{Rational frieze implies Dynkin diagram}

\begin{theorem} \label{theorem1}
Let $Q$ be a quiver such that the underlying undirected graph is connected. The sequences on the frieze associated to $Q$ are simultaneously bounded or unbounded. Suppose that the frieze is rational. If the sequences are all bounded (resp. all unbounded), then $Q$ is a Dynkin diagram (resp. an extended Dynkin diagram ) with some acyclic orientation. 
\end{theorem}

The theorem is proved by studying the asymptotics of the sequences; these are computed by the exponential polynomial. Then the additive or subadditive function is obtained by taking the logarithm. It appears that the recursion formula (\ref{friserecursion}) is some multiplicative analogue of the additivity or subadditivity formula.

We only prove the theorem under the stronger assumption that the frieze is $\mathbb N$-rational. The complete result is proved in \cite{ARS}.

\begin{proof}
1. Formula (\ref{friserecursion}) may be rewritten as
\begin{equation}\label{recursionbis}
\quad a(v,n+1)a(v,n)=1+\prod_{v\rightarrow w}a(w,n)\prod_{w\rightarrow v}a(w,n+1)  .
\end{equation}
Recall that by the Laurent phenomenon, these numbers are all positive natural numbers. 
Hence if $a(v,n)$ is bounded, then $a(w,n)$ is bounded for each neighbour $w$ of $v$. Since the graph is connected, if one sequence is bounded, then each sequence is bounded.

2. We now assume that all the sequences are bounded. Since they are integer-valued, they take only finitely many values. Since they satisfy linear recursions, they are ultimately periodic. Let $p$ be a common period and let $n_0$ be such that each sequence is purely periodic for $n\geq n_0$.

Let $b(v)=\prod_{n_0\leq n <n_0+p} a(v,n)$. Note that $b(v)>1$. Indeed, if $a(v,n)=1$, then $a(v,n+1)>1$ by Eq.(\ref{friserecursionbis}); moreover, each $a(v,n)$ is a positive integer. We have, since $a(v,n_0)=a(v,n_0+p)$, $b(v)=\prod_{n_0\leq n <n_0+p} a(v,n+1)$. Thus
\begin{eqnarray*}
b(v)^2 &=& \left(\prod_{n_0\leq n <n_0+p} a(v,n)\right)\left( \prod_{n_0\leq n <n_0+p} a(v,n+1)\right) \\
&=& \prod_{n_0\leq n <n_0+p} a(v,n)a(v,n+1) \\
&=& \prod_{n_0\leq n <n_0+p} (1+\prod_{v\rightarrow w}a(w,n)\prod_{w\rightarrow v}a(w,n+1))
\end{eqnarray*}
by Eq.(\ref{recursionbis}). Thus
\begin{eqnarray*}
b(v)^2 &>& \prod_{n_0\leq n <n_0+p} (\prod_{v\rightarrow w}a(w,n)\prod_{w\rightarrow v}a(w,n+1) )\\
&=&  ( \prod_{v\rightarrow w}\prod_{n_0\leq n <n_0+p}a(w,n) )(\prod_{w\rightarrow v}\prod_{n_0\leq n <n_0+p} a(w,n+1))\\
&=& ( \prod_{v\rightarrow w} b(w))(\prod_{w\rightarrow v} b(w) ) \\
&=& \prod _{\{v,w\}\in E} b(w) ,
\end{eqnarray*}
where $E$ is the set of edges of the underlying graph.
Taking logarithms, we obtain
$$
2 log(b(v)) > \sum _{\{v,w\}\in E} log(b(w)) 
$$
and we have a subadditive function which is not additive, since $b(v)>1$.

3. We now assume that all the sequences are unbounded. We apply Corollary \ref{equivalent}: there exist a positive natural number $p$, and for $i=0,\ldots,p$ and $v\in V$, natural numbers $e(v,i)$ and $\lambda(v,i)\in\mathbb R_{\geq1}$ such that
$$a(v,i+pn)\approx n^{e(i,v)}\lambda(v,i)^{n}
$$
when $n\rightarrow \infty$.

Define
$$b(v,n)=a(v,n)a(v,n+1)\cdots a(v,n+p-1).$$ 
We have

$$
b(v,pn)=a(v,pan)a(v,pn+1)\cdots a(v,pn+p-1) $$ $$
\approx \lambda(v,0)^{n} n^{e(v,0)}\cdots \lambda(v,p-1)^{n} n^{e(v,p-1)}\approx \lambda(v)^{n} n^{f(v)}
$$
so that we have
\begin{equation}\label{bjpnk}
b(v,pn)\approx \lambda(v)^{n} n^{f(v)}
\end{equation}

where $\lambda(v)=\lambda(v,0)\cdots\lambda(v,p-1)$ and $f(v)=e(v,0)+\cdots+e(v,p-1)$. Now, by rearrangement,
$$b(v,pn)^2=a(v,pn)a(v,pn+1)a(v,pn+1)a(v,pn+2)\cdots $$
$$ \cdots a(v,pn+p-1)a(v,pn).
$$
By Corollary \ref{equivalent}, $a(v,pn)\approx a(v,pn+p)$. Thus
$$b(v,pn)^2 \approx\prod_{0\leq i<p}a(v,pn+i)a(v,pn+i+1).$$
Using Eq.(\ref{recursionbis}), we obtain
$$
b(v,pn)^2\approx \prod_{0\leq i<p} (1+\prod_{v\rightarrow w} a(w,pn+i)\prod_{w\rightarrow v} a(w,pn+i+1)).
$$
Let $$u_n=\prod_{v\rightarrow w} a(w,pn+i)\prod_{w\rightarrow v} a(w,pn+i+1).$$ 

Now comes a technicality. By Corollary \ref{equivalent}, $a(w,pn+i) \approx \lambda(w,i)^n n^{e(w,i)}$ with $\lambda(w,i)\geq 1$ and $e(w,i)\geq 0$. Hence if $u_n$ is unbounded, there exists $w$ with: either $v\rightarrow w$, and $\lambda(w,i)>1$ or $e(w,i)\geq 1$; or $w\rightarrow v$, and $\lambda(w,i+1)>1$ or $e(w,i+1)\geq 1$. Then $\lim_{n\rightarrow\infty}u_n=\infty$ and $u_n\approx 1+u_n$. Otherwise, $u_n$ is bounded and by Corollary \ref{equivalent}, $u_n$ has a finite limit, since we must have $\lambda(w,i)=1$ and $e(w,i)=0$; therefore $u_n\approx 1+u_n$. Thus in both cases, $1+u_n\approx u_n$.

We deduce that 
$$
b(v,pn)^2 
\approx\prod_{0\leq i<p} \prod_{v\rightarrow w} a(w,pn+i)\prod_{w\rightarrow v} a(w,pn+i+1)$$
$$\approx\prod_{0\leq i<p}\prod_{v\rightarrow w}\lambda(w,i)^{n} n^{e(w,i)}\prod_{w\rightarrow v} \lambda(w,i+1)^{n} n^{e(w,i+1)} $$
$$\approx\prod_{v\rightarrow w}  \prod_{0\leq i<p} \lambda(w,i)^{n} n^{e(w,i)}\prod_{w\rightarrow v}  \prod_{0\leq i<p} \lambda(w,i+1)^{n} n^{e(w,i+1)}.$$
Since $e(w,0)=e(w,p)$ and $\lambda(w,0)=\lambda(w,p)$, we obtain
\begin{eqnarray*}
b(v,pn)^2 &\approx& (\prod_{v\rightarrow w} \lambda(w)^{n} n^{f(w)}) (\prod_{w\rightarrow v} \lambda(w)^{n} n^{f(w)}) \\
&\approx& \prod_{\{v, w\}\in E}\lambda(w)^{n} n^{f(w)}.
\end{eqnarray*}
Thus by Eq.(\ref{bjpnk}),
$$\lambda(v)^{2n}n^{2f(v)}\approx\prod_{\{v, w\}\in E}\lambda(w)^{n} n^{f(w)}.$$
Therefore, 
$$\lambda(v)^2=\prod_{\{v, w\}\in E}\lambda(w)
$$
and
$$
2f(v)=\sum_{\{v, w\}\in E}f(w).
$$
If the $f(v)$ are all positive, we have the additive function $f$. If one of them is 0, then they are all 0, by connectedness of the graph and the above equation. In this case, $e(v,i)=0$ for any $v$ and $i$. Thus Corollary \ref{equivalent}† ensures that for any $v$, some $\lambda(v,i)>1$ and therefore $\lambda(v)>1$. Taking logarithms, we find
$$
2log(\lambda(v))=\sum_{\{v, w\}\in E}log(\lambda(w)))
$$
and we have therefore an additive function.

In case 2 and 3, we conclude by using Th.\ref{Dynkin}.
\end{proof}
 
\section{Rationality for Dynkin diagrams  of type ${\mathbb A}$ and $\tilde {\mathbb A}$}

It follows from the work of Fomin and Zelevinsky \cite{FZ} that if $Q$ is a quiver of Dynkin type (that is, obtained by orienting the edges of a Dynkin diagram), then the sequences of its frieze are periodic. They are therefore obviously $\mathbf N$-rational.

\begin{theorem}\label{tilde A}
Let $Q$ be an acyclic quiver whose underlying undirected graph is an extended diagram of type $\tilde {\mathbb A}_m$ (that is a cycle). Then its frieze is $\mathbf N$-rational.
\end{theorem}

The theorem is proved by associating to such a quiver $Q$ an $SL_2$-tiling associated to a purely periodic frontier. Then applying Theorem \ref{rational}, since the sequences of the frieze appear on the diagonals of the tiling.

\begin{proof}
Let $1,...,m+1$ be the vertices of the graph $\tilde {\mathbb A}_m$, with edges $\{j,j+1\}$, $j=1,...,m+1$, with $j+1$ taken $mod \, m+1$. An acyclic orientation being given, let $x_j=x$ if the orientation is $j\rightarrow j+1$ and $x_j=y$ if it is $j\leftarrow j+1$. Let $a(j,n)$ be the sequences of the frieze, $j=1,\ldots,m+1$. We extend the notation $a(j,n)$ to $j\in \mathbf Z$ by taking $j$ $mod. m+1$. Let $w$ be the word $x_1\cdots x_{m+1}$, which encodes the orientation. Then $^\infty w^\infty$ is an admissible frontier; indeed, $x$ and $y$ appear both in $w$, since the orientation is acyclic. Embed this frontier into the plane and denote by $P_j$, with $j\in \mathbf Z$, the successive points of this embedding, in such a way that  $P_j$ corresponds to the point between $x_{j-1}$ and $x_j$, with $j$ taken $mod.\, m+1$. Let $t$ be the $SL_2$-tiling given by Theorem \ref{tiling-variables}. Because of the periodicity of the frontier, the diagonal ray $b(j,n)$ of origin $P_j$, $j\in \mathbf Z$, depends only on the class of $j$ $mod \, m+1$.

We claim that the ray $b(j,n)$ is equal to $a(j,n)$, for $j=1,..,m+1$. This is true for $n=0$, since both are equal to 1. It is enough to show that $b(j,n)$ satisfies Eq. (\ref{friserecursion}). Fix $j=1,...,m+1$. We have four cases according to the relative positions of $P_{j-1},P_j,P_{j+1}$ .
They correspond to the four possible values of the couple $(x_{j-1},x_j)$: $$ (x,x), (x,y), (y,x),(y,y).$$
By definition of $w$, these four cases correspond to the four possible orientations:
$$j-1\rightarrow j\rightarrow j+1, \,j-1\rightarrow j\leftarrow j+1,$$
$$\, j-1\leftarrow j\rightarrow j+1,\, j-1\leftarrow j\leftarrow j+1.
$$
Thus, by Eq.(\ref{friserecursion}), they correspond to the four induction formulas $a(j,n+1)=\frac{1+A}{a(j,n)}$, where $A$ takes one of the four possible values:
$$
a(j-1,n+1)a(j+1,n), \quad a(j-1,n+1)a(j+1,n+1), $$
$$a(j-1,n)a(j+1,n), \quad a(j-1,n)a(j+1,n+1).
$$
Regarding the tiling, these four cases correspond to the four possible configurations, shown in Figure \ref{config}.

\begin{figure}
$$
\begin{array}{lllllllllllllll}
P_{j-1}&P_j&P_{j+1} \\
&.&.&. \\
&&.&.&. \\
&&&.&.&. \\
&&&&b(j-1,n)&b(j,n)&b(j+1,n) \\
&&&&&b(j-1,n+1)&b(j,n+1)\\
\\
\\
&P_{j+1} \\
P_{j-1}&P_j&. \\
&.&.&. \\
&&.&.&b(j+1,n) \\
&&&b(j-1,n)&b(j,n)&b(j+1,n+1) \\
&&&&b(j-1,n+1)&b(j,n+1)
\\
\\
P_j&P_{j+1} \\
P_{j-1}&.&. \\
&.&.&. \\
&&.&b(j,n)&b(j+1,n) \\
&&&b(j-1,n)&b(j,n+1) \\
\\
\\
P_{j+1} \\
P_j &.\\
P_{j-1}&.&. \\
&.&.&b(j+1,n) \\
&&.&b(j,n)&b(j+1,n+1) \\
&&&b(j-1,n)&b(j,n+1) \\
\end{array}
$$
\caption{The 4 possible configurations}\label{config}
\end{figure}

Hence, by the $SL_2$-condition, they correspond to the four induction formulas for $b(j,n)$: $b(j,n+1)=\frac{1+B}{b(j,n)}$, where $B$ takes one of the four possible values:
$$
 b(j-1,n+1)b(j+1,n), b(j-1,n+1)b(j+1,n+1),$$
$$ b(j-1,n)b(j+1,n), b(j-1,n)b(j+1,n+1).
$$
We deduce that $a(j,n)=b(j,n)$ for ant $j$ and $n$. This concludes the proof, by using Theorem \ref{rational}.

\end{proof}

The proof is illustrated in the frieze shown in Figure \ref{friseA2}, of type $\tilde A_2$. The corresponding tiling is shown in Figure \ref{tiling}.

\section{Further properties of $SL_2$-tilings}

\subsection{Tameness and linearization coefficients}\label{tameness}

\begin{proposition}\label{threecolumns}
Given three successive columns $C_0,C_1,C_2$ of a tame $SL_2$-tiling $t$, there is a unique coefficient $\alpha$ such that 
\begin{equation}\label{lin}
C_0+C_2=\alpha C_1.
\end{equation}
\end{proposition}

\begin{proof} Consider two consecutive rows, with elements $a,b,c$ in the first row and in columns $C_0,C_1,C_2$ respectively, and elements $d,e,f$ in the second row; now let $x,y,z$ be elements of these 3 columns, respectively, located on some arbitrary row; then the 3 by 3 matrix $\left(\begin{array}{ccc}x&y&z\\a&b&c\\d&e&f\end{array}\right)$ is a submatrix of the tiling; hence its determinant is 0, so that, by expanding the determinant with respect to the first row, and noting that the 2 by 2 adjacent minors are equal to 1, we obtain that $x+z=\alpha y$, with $\alpha =det \left(\begin{array}{cc}a&c\\d&f\end{array}\right)$. Hence Eq.(\ref{lin}) holds.
\end{proof}

We call $\alpha$ the {\em linearization coefficient} of column $C_1$. A similarly definition applies for rows. 

The next result shows that the $SL_2$-tiling associated to a frontier with variables all equal to 1 is tame; moreover the linearization coefficients are easily computed. In order to understand the statement, note that each column intersects the frontier  and that this intersection is a finite number of points of $\mathbb Z^2$; for example, in the tiling shown in Figure \ref{tilingN}, the column containing the number 29 has 4 intersection points with the frontier.

\begin{corollary}\label{lincoeff}
Let $t$ be the $SL_2$-tiling $t$ associated to some frontier; then $t$ is tame.  Let $C$ be a column of $t$;  let $k$ be the number of common points in this column and the frontier. Then the linearization coefficient of $C$ is $k+1$.
\end{corollary}

\begin{proof}
Let $C_i$, $i=0,1,2$, be three successive columns with $C=C_1$ and $\alpha$ the linearization coefficient of $C$. Then there exists a row such that the numbers on this row and on these 3 columns are respectively $1,1,k$; see the example above. Then Eq.(\ref{lin}) gives $1+k=\alpha$, which was to be shown.

It remains to show that $t$ is tame. Let $P,Q,R$ be three points in $\mathbb Z^2$ on the same horizontal line. We have to show that $t(P)+t(R)=(k+1)t(Q)$, where $k$ is defined above, with $Q$ on column $C$. Suppose that the points are all below the frontier. Then the words associated by projection on the frontier of $P,Q,R$ are respectively of the form $w,wy^jx,wy^jxy^{k-1}x$, see Figure \ref{PQR}.

\begin{figure}
$$
\begin{array}{cccccccccccc}
&&&&&&&&&\bf 1&\bf 1 \\
&&&&&&&&&\vdots&  \\
&&&&&&&&&\vdots&  \\
&&&&&&&&\bf 1&\bf 1&  \\
&&&&&&&\iddots&&\\
&&&&&&w&&&&  \\
&&&&&\iddots&&&&&    \\
&&&&\bf 1&&&&P&Q&R\\

\end{array}
$$
\caption{}\label{PQR}
\end{figure}

We have $$\mu(y^jx)=\left(\begin{array}{cc}1&0\\j&1\end{array}\right)\left(\begin{array}{cc}1&1\\0&1\end{array}\right)=\left(\begin{array}{cc}1&1\\j&j+1\end{array}\right).$$

Let    
$\mu(w)=\left(\begin{array}{cc}a&b\\c&d\end{array}\right)$. Then $$\mu(wy^jx)=\left(\begin{array}{cc}a&b\\c&d\end{array}\right)\left(\begin{array}{cc}1&1\\j&j+1\end{array}\right)  
=\left(\begin{array}{cc}a+jb&a+(j+1)b\\c+jd&c+(j+1)d\end{array}\right)$$ 
and 
$$\mu(wy^jxy^{k-1}x)=\left(\begin{array}{cc}a+jb&a+(j+1)b\\c+jd&c+(j+1)d\end{array}\right)\left(\begin{array}{cc}1&1\\k-1&k\end{array}\right)$$ 
$$=\left(\begin{array}{cc }\times&\times\\\times&c+jd+kc+k(j+1)d\end{array}\right).$$ 


Thus by Cor.\ref{t(P)=mu(w)}, we have $t(P)=d$, $t(Q)=c+(j+1)d$ and $t(R)=c+jd+kc+k(j+1)d$. Thus $t(P)+t(R)=(k+1)t(Q)$.

If the points are all above the frontier, the argument is symmetric. Then remaining case is when $P$ is at the left of the frontier, $Q$ on the frontier and $R$ on the right. Then is is easily seen that $t(P)=i$, $t(Q)=1$ and $t(R)=j$ with $i+j=k+1$, which ends the proof.
\end{proof}

Let $C,C'$ be two columns of the $SL_2$-tiling $t$, not necessarily adjacent. We call {\em semi-adjacent }2 {\em by} 2 {\em minor associated to these columns} a 2 by 2 minor of $t$ constructed on these columns and on two adjacent rows. If the columns are themselves adjacent, then the minor equals 1 and is therefore independent of the chosen adjacent rows.
We see in the next result that this independence is still true for any columns.

\begin{proposition}\label{semi-adjacent}
Let $C,C'$ be two columns of the $SL_2$-tiling $t$. Then the semi-adjacent $2$ by $2$ minor associated to these columns is independent of the choice of the chosen adjacent rows.
\end{proposition}

This will be proved in the next section.

\subsection{Continuants polynomials}
The {\em signed continuant polynomials} are defined as follows. Let $a_1,\ldots,a_n$ be elements of some ring. Define for $n\geq 1$, 
\begin{equation}\label{RecursionContinuant}
q_n(a_1,\ldots,a_n)=q_{n-1}(a_1,\ldots,a_{n-1})a_n-q_{n-2}(a_1,\ldots,a_{n-2}),
\end{equation}
setting $q_{-1}:=0$ and $q_0:=1$. We omit indices when possible, writing simply $q(a_1,\ldots,a_n)$ for $q_n(a_1,\ldots,a_n)$.                       
Let us now consider the particular $SL_{2}$ matrices
    $${Q(a):=\begin{pmatrix} 0 & -1\\ 1 & a \end{pmatrix}}.$$
One has the following result, see \cite{BeR}  8.1.  

\begin{lemma} 
\begin{equation}\label{matricecontinuants}
     {Q(a_1) Q(a_2) \cdots Q(a_n)= \begin{pmatrix} -q(a_2,\ldots,a_{n-1}), & -q(a_2,\ldots,a_{n})\\
              q( a_1,\ldots,a_{n-1}), & q(a_1,\ldots,a_{n})\end{pmatrix}}.
\end{equation}
\end{lemma}
The proof of this lemma is left to the reader. We have therefore $$Q(a_{i+1})\ldots Q(a_n)=(Q(a_1)\ldots Q(a_i))^{-1}Q(a_1)\ldots Q(a_n).$$ Since the matrix to be inverted is in $SL_2$, by taking the 2,2-entry of both sides, we obtain, when the ring is commutative,

\begin{equation}\label{ident}
q(a_{i+1},\ldots,a_n)=q(a_1,\ldots,a_{i-1})q(a_2,\ldots,a_n)-q(a_2,\ldots,a_{i-1})q(a_1,\ldots,a_n).
\end{equation}

\begin{lemma}\label{linearcontinuant}
Let $t$ be a tame $SL_2$-tiling of the plane and $C_0,\ldots,C_{n+1}$ successive columns of $t$, with linearization coefficients $\alpha_0,\ldots,\alpha_{n+1}$. Then for any $i$ in $\{1,\ldots,n\}$
$$
q(\alpha_{i+1},\ldots,\alpha_n)C_0+q(\alpha_1,\ldots,\alpha_{i-1})C_{n+1}=q(\alpha_1,\ldots,\alpha_n)C_i.
$$
\end{lemma}

\begin{proof}
We use the identity 
\begin{equation}\label{columns}
C_j=-q(\alpha_2,\ldots, \alpha_{j-1})C_0+q(\alpha_1,\ldots,\alpha_{j-1})C_1,
\end{equation}
which is proved as follows. First, we have by definition of the linearization coefficients,
$C_{j+2}=-C_j+\alpha_{j+1}C_{j+1}$. This implies that $(C_j,C_{j+1})Q(\alpha_{j+1})=(C_{j+1},C_{j+2})$. It follows that for any natural number $j$, one has 
\begin{equation}\label{matrixcolumns}
(C_0,C_1)Q(\alpha_1)\ldots Q(\alpha_j)=(C_j,C_{j+1}),
\end{equation}
We conclude by using Eq.(\ref{matricecontinuants}).

Suppose first that $i=1$. Then, with Eq.(\ref{columns}),
$$
q(\alpha_{2},\ldots,\alpha_n)C_0+C_{n+1}=q(\alpha_{2},\ldots,\alpha_n)C_0-q(\alpha_2,\ldots, \alpha_{n})C_0+q(\alpha_1,\ldots,\alpha_{n})C_1$$
$$=q(\alpha_1,\ldots,\alpha_{n})C_1,
$$
which proves the identity of the lemma for $i=1$. Suppose now that $i>1$. We obtain, with Eq.(\ref{columns}), Eq.(\ref{ident}) and Eq.(\ref{columns}) again,
$$
q(\alpha_{i+1},\ldots,\alpha_n)C_0+q(\alpha_1,\ldots,\alpha_{i-1})C_{n+1}$$
$$=q(\alpha_{i+1},\ldots,\alpha_n)C_0+q(\alpha_1,\ldots,\alpha_{i-1})(-q(\alpha_2,\ldots, \alpha_{n})C_0+q(\alpha_1,\ldots,\alpha_{n})C_1)
$$
$$=(q(\alpha_{i+1},\ldots,\alpha_n)-q(\alpha_1,\ldots,\alpha_{i-1})q(\alpha_2,\ldots, \alpha_{n}))C_0$$
$$+q(\alpha_1,\ldots,\alpha_{i-1})q(\alpha_1,\ldots,\alpha_{n})C_1
$$
$$=q(\alpha_2,\ldots,\alpha_{i-1})-q(\alpha_1,\ldots,\alpha_n)C_0+q(\alpha_1,\ldots,\alpha_{i-1})q(\alpha_1,\ldots,\alpha_{n})C_1
$$
$$=q(\alpha_1,\ldots,\alpha_{n})(-q(\alpha_2,\ldots,\alpha_{i-1})C_0+q(\alpha_1,\ldots,\alpha_{i-1})C_1)
$$
$$=q(\alpha_1,\ldots,\alpha_{n})C_i.
$$
\end{proof}

\begin{proof} (Proposition \ref{semi-adjacent})
Let $\alpha_1,\ldots,\alpha_n$ be the linearization coefficients of the successive columns lying strictly between $C$ and $C'$. We show that the semi-adjacent 2 by 2 minor associated to $C,C'$ and to two adjacent rows is equal to the continuant polynomial $q(\alpha_1,\ldots,\alpha_n)$. This will prove the proposition. Denote by $C_1,\ldots,C_{n}$ the columns of length 2 obtained by restricting the $n$ columns above to the given rows, and similarly denote $C_0,C_{n+1}$ the same restriction for $C$ and $C'$.

Now we have by Lemma \ref{linearcontinuant}, $q(\alpha_2,\ldots,\alpha_n)C_0+C_{n+1}=q(\alpha_1,\ldots,\alpha_n)C_1$.
Thus $$\det(C_0,C_{n+1})=q(\alpha_1,\ldots,\alpha_n)\det(C_0,C_1)=q(\alpha_1,\ldots,\alpha_n),$$ since $t$ is an $SL_2$-tiling and therefore $det(C_0,C_1)=1$.
This ends the proof.
\end{proof}

\subsection{The path model}\label{pathmodel}
Consider an $SL_2$-tiling $t$ associated to an admissible frontier, as in Theorem \ref{tiling-variables}. We call {\em fringe} the set of points on the path together with their translate by the vector $(1,1)$. We consider below discrete paths with steps $(1,0)$ and $(0,-1)$, cf. Figure \ref{coord}.

\begin{proposition}\label{paths}
Let $P$ be a point below the frontier and $A,B$ be its horizontal and vertical projection on the frontier. Then $t(P)$ is equal to the number of paths from $A$ to $B$ which remain in the fringe.
\end{proposition}

As an example, look at the figure below. Consider the paths from $A$ (the southwest 1) to $B$ (the northeast 1) that remain in the fringe; one of them passes through 7, and the others pass through $M$ (the boldfaced 1): they decompose as a product of two paths, with 3 choices for the first and 2 for the second; hence there are $1+2\times 3=7$ paths. 

$$\begin{array}{rrrrrrrrrrrrrrrr}
&&1&1\\
1&1&\bf 1&2\\
1&2&3&7\\
&&&&&&&&&
\end{array}
$$

The proposition may be proved directly (Exercise \ref{pathex}). A proof may be found in \cite{BeR} (Proposition 7): it uses the Gessel-Viennot theory of non-intersecting paths and corresponding minors.

\subsection{Quadratic forms}\label{quad}

Given a matrix $M=\left(\begin{array}{cc}a&b\\c&d\end{array}\right)$, associate to it as usually the bilinear form, in general not symmetric,
\begin{equation}\label{form}
B((x,y),(x',y'))=(x,y)M\left(\begin{array}{cc}x'\\y'\end{array}\right)=axx'+bxy'+cyx'+dyy'
\end{equation}
and the associated quadratic form
$$Q(x,y)=B((x,y),(x,y))=(x,y)M\left(\begin{array}{cc}x\\y\end{array}\right)=ax^2+(b+c)xy+dy^2.$$
Recall that the {\em discriminant} of $Q(x,y)$ is the discriminant of the degree 2 equation $Q(x,1)=0$, that is, $(b+c)^2-4ad$. If $M$ is of determinant 1, we have $ad=1+bc$ so that the discriminant of $Q$ is $b^2+2bc+c^2-4bc-4=(b-c)^2-4$.

Given a finite or infinite word $w$ on the alphabet $\{x,y\}$, we call {\em transpose} of $w$, and denote it by $^tw$, the word obtained by reversing it and exchanging $x$ and $y$. For instance, $^t(xyyxy)=xyxxy$. If $w$ is a right infinite word, then its transpose is a left infinite word.

We consider an admissible frontier of the form ${}^tsws$, where $w$ (resp. $s$) is a finite (resp. right infinite) word on $\{x,y\}$. We shall see that certain values of the corresponding tiling $t$ are values of the quadratic form associated to $\mu(w)$. 

\begin{figure}
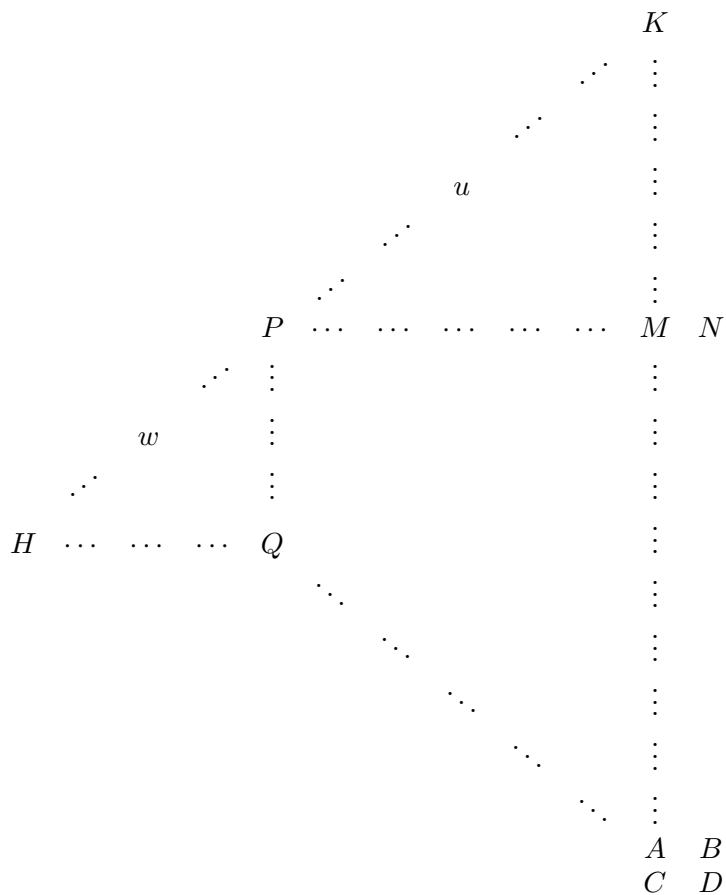

$$
\begin{array}{cccccccccccccccccccc}
&&&&&&&&&&K\\
&&&&&&&&&\iddots&\vdots\\
&&&&&&&&\iddots&&\vdots\\
&&&&&&&u&&&\vdots\\
&&&&&&\iddots&&&&\vdots  \\
&&&&&\iddots&&&&&\vdots  \\
&&&&P&\cdots&\cdots&\cdots&\cdots&\cdots&M&N  \\
&&&\iddots&\vdots&&&&&&\vdots\\
&&w&&\vdots&&&&&&\vdots  \\
&\iddots&&&\vdots&&&&&&\vdots   \\
H&\cdots&\cdots&\cdots&Q&&&&&&\vdots  \\
&&&&&\ddots&&&&&\vdots\\
&&&&&&\ddots&&&&\vdots\\
&&&&&&&\ddots&&&\vdots&\\
&&&&&&&&\ddots&&\vdots&&\\
&&&&&&&&&\ddots&\vdots&&&\\
&&&&&&&&&&A&B \\
&&&&&&&&&&C&D\\
\end{array}
$$
\caption{The chosen points on the plane}\label{points}
\end{figure}

We define for this several points determined by the frontier. First, let $Q$ be the lower right corner of the smallest rectangle containing the part $w$ of the frontier; let $P$ the upper right corner of this rectangle and $H$ its lower left corner, so that $w$ is the path from $H$ to $P$. Let $M,N$ be two successive points on the horizontal half-line starting from $P$ to the right, and $A,D$ be two successive points on the diagonal half-line starting from $Q$ to the south-east, with $A$ and $M$ on the same vertical line; let $B,C$ the points such that $A,B,C,D$ form an adjacent 2 by 2 submatrix of $t$. Let $u$ be the word associated to $M$ and $v$ the word associated to $N$. See Figure \ref{points}, where $w$ is the path on the frontier from $H$ to $P$ and $u$ from $P$ to $K$.

\begin{proposition}
Let $M=\mu(w)$ and consider the quadratic form $Q$ and the bilinear form $B$ associated to $M$ as above. Let $\mu(u)=\left(\begin{array}{cc}p&q\\r&s\end{array}\right)$ and $\mu(v)=\left(\begin{array}{cc}p'&q'\\r'&s'\end{array}\right)$.
One has 
$$\left(\begin{array}{cc}t(A)&t(B)\\t(C)&t(D)\end{array}\right)=\left(\begin{array}{cc}Q(q,s)&B((q,s),(q',s'))\\B((q',s'),(q,s))&Q(q',s')\end{array}\right).$$
Moreover $b-c=t(C)-t(B)$, $s=t(M)$ and $q$ is equal to the semi-adjacent $2$ by $2$ minor determined by the columns of $P$ and that of $M$. 
\end{proposition}

As an example, look at the tiling below: the word $w=yx^2yx$ in the frontier has been indicated by the boldfaced 1's. The rightmost such $1$ is the point $P$. The matrix $\mu(w)$ is equal to $\left(\begin{array}{cc}a&b\\c&d\end{array}\right)=\left(\begin{array}{cc}3&5\\4&7\end{array}\right)$. The quadratic form is $Q(x,y)=3x^2+9xy+7y^2$. Point $Q$ (with 7) and its diagonal are boldfaced; consider the point A corresponding to 691: it is equal to Q(3,8) and one has $1129-1128=1=b-c$.
The point $M$ is the point labelled $8$ in the same column as 691. The word $u$ here is equal to $yyxyx$ and one has $\left(\begin{array}{cc}p&q\\r&s\end{array}\right)=\mu(u)=\left(\begin{array}{cc}2&3\\5&8\end{array}\right)$. One has $s=8$ and $q=3=det\left(\begin{array}{cc}1&5\\1&8\end{array}\right)$.

$$
\begin{array}{cccccccccccccccccccccccccc}
277&117&74&31&19&7&2&1&1&1\\
116&49&31&13&8&3&1&1&2&3\\
71&30&19&8&5&2&1&2&5&8\\
26&11&7&3&2&\bf 1&\bf 1&3&8&13\\
7&3&2&\bf 1&\bf1&\bf1&2&7&19&31\\
2&1&1&\bf1&2&3&\bf7&25&68&111\\
1&1&2&3&7&11&26&\bf93&253&413\\
1&2&5&8&19&30&71&254&\bf691&1128\\
1&3&8&13&31&49&116&415&1129&\bf1843\\
\end{array}
$$

\begin{proof}
By inspection of the figure, it is seen that the words associated to $A,B,C,D$ are ${}^tuwu$, ${}^tuwv$, ${}^tvwu$ and ${}^tvwv$. Thus we have, since $\mu({}^tu)={}^t\mu(u)$,
$$
t(A)=\mu({}^tuwu)_{22}=(0,1){}^t\mu(u)\mu(w)\mu(u)\left(\begin{array}{cc}0\\1\end{array}\right)
$$
$$={}^t(\mu(u) \left(\begin{array}{cc}0\\1\end{array}\right) )M\mu(u)\left(\begin{array}{cc}0\\1\end{array}\right)=(q,s)M\left(\begin{array}{cc}q\\s\end{array}\right)=Q(q,s).$$

Similarly,
$$t(B)=\mu({}^tuwv)_{22}=(0,1){}^t\mu(u)\mu(w)\mu(v)\left(\begin{array}{cc}0\\1\end{array}\right)
$$
$$={}^t(\mu(u) \left(\begin{array}{cc}0\\1\end{array}\right) )M\mu(v)\left(\begin{array}{cc}0\\1\end{array}\right)=(q,s)M\left(\begin{array}{cc}q'\\s'\end{array}\right)=B((q,s),(q',s')).$$

The calculations for $t(C)$ and $t(D)$ are similar. We have by Eq.(\ref{form}, $t(C)-t(B)=B((q',s'),(q,s))-B((q,s),(q',s'))=bq's+cs'q-bqs'-csq'=(b-c)(q's-qs').$ Now, since $M$ and $N$ are adjacent points, we have $v=uy^lx$ for some natural number $l$; thus $$\left(\begin{array}{cc}p'&q'\\r'&s'\end{array}\right)=\mu(v)=\mu(u)\mu(y)^l\mu(x)=\mu(u)\left(\begin{array}{cc}1&0\\l&1\end{array}\right)\left(\begin{array}{cc}1&1\\0&1\end{array}\right)$$
$$=\left(\begin{array}{cc}p&q\\r&s\end{array}\right)\left(\begin{array}{cc}1&1\\l&l+1\end{array}\right)=\left(\begin{array}{cc}p+lq&p+(l+1)q\\r+ls&r+(l+1)s\end{array}\right).$$ Thus $q's-qs'=(p+(l+1)q)s-q(r+(l+1)s)=ps-qr=1$, since $\mu(u)$ has determinant 1.

It remains to prove the assertion about the semi-adjacent minor. We show first that if M,N,P are three successive points on the same horizontal line on the plane, below the frontier, then $q(M)+q(P)=\alpha q(N)$, where $\alpha$ is the linearization coefficient of the column of $N$ and where $q(M)=\mu(u)_{12}$ with $u$ the word associated to $M$. The words associated to $M,N,P$ are $u,uy^lx,uy^lxy^mx$, for some natural numbers $l,m$ and $\alpha=m+2$ by Cor.\ref{lincoeff}. With $\mu(u)=\left(\begin{array}{cc}p&q\\r&s\end{array}\right)$, a simple computation then shows that $q(M)=q$, $q(N)=p+(l+1)q$ and $q(P)=p+lq+(p+(l+1)q)(m+1)$. Thus
$q(M)+q(P)=\alpha q(N)$ or equivalently $q(P)=q(N)\alpha-q(M)$. 

Recall from the proof of Prop.\ref{semi-adjacent} that the semi-adjacent minor of two columns is equal to the signed continuant polynomial of the linearization coefficients of the columns lying strictly between them. The similarity of the previous formula and the recursion formula (\ref{RecursionContinuant}) for the signed continuant polynomials then implies the result.
\end{proof}

\begin{corollary}\label{square}
Suppose that $w=yx^hy$ for some natural number $h$. Let $M'$ (resp. $N'$) be the point immediately below $M$ (resp. $N$). Then $t(A)=(h+1)t(M')^2$ and $t(B)-t(C)=2$. Moreover $t(B)=(h+1)t(M')t(N')+1$, $t(C)=(h+1)t(M')t(N')-1$.
\end{corollary}

As an example, see Figure \ref{tilingN}: we have $h=0$, $w=yy$ (shown by boldfaced 1's on the figure), $Q$ is the lowest point with boldfaced 1, the points $A,B,C,D$ are respectively the ones with $8^2,89,87,11^2$ and the points $M',N'$ are those with $8,11$, in the same column as $8^2$ and $11^2$ respectively.

\begin{proof}
Note that, since the last letter of $w$ is $y$, the words associated to $M',N'$ are respectively $yu$ and $yuy^lx$. We have $$\mu(w)=\mu(y)\mu(x)^h\mu(y)=\left(\begin{array}{cc}1&0\\1&1\end{array}\right)\left(\begin{array}{cc}h+1&h\\1&1\end{array}\right)=\left(\begin{array}{cc}h+1&h\\h+2&h+1\end{array}\right).$$
Thus the quadratic form is $Q(x,y)=(h+1)x^2+(2h+2)xy+(h+1)y^2=(h+1)(x+y)^2$.
This proves that $t(A)=(h+1)(q+s)^2$, with $\mu(u)=\left(\begin{array}{cc}p&q\\r&s\end{array}\right)$. 

Now $t(M')$ is the 2,2-entry of the matrix $$\mu(yu)=\left(\begin{array}{cc}1&0\\1&1\end{array}\right)\left(\begin{array}{cc}p&q\\r&s\end{array}\right)
=\left(\begin{array}{cc}p&q\\p+r&q+s\end{array}\right).$$ Thus $t(A)=(h+1)t(M')^2$.
Moreover, $t(N')$ is the 2,2,-entry of the matrix
$$\mu(yuy^lx)=\left(\begin{array}{cc}p&q\\p+r&q+s\end{array}\right)\left(\begin{array}{cc}1&1\\l&l+1\end{array}\right)$$
$$=\left(\begin{array}{cc}p+ql&p+q(l+1)\\p+r+(q+s)l&p+r+(q+s)(l+1)\end{array}\right).$$
Thus we obtain $t(N')=p+r+(q+s)(l+1)$. By the previous proof $t(B)=B((q,s),(q',s'))=qaq'+qbs'+scq'+sds'=qa(p+(l+1)q)+qb(r+(l+1)s)+sc(p+(l+1)q)+sd(r+(l+1)s)$ $=qap+qbr+scp+sdr+(l+1)(qaq+qbs+scq+sds)$. Thus $t(B)-(h+1)t(M')t(N')=qap+scp+qbr+sdr+(l+1)(qaq+scq+qbs+sds)-(h+1)(q+s)(p+r+(q+s)(l+1))=(h+1)pq+(h+2)ps+hqr+(h+1)rs+(l+1)((h+1)q^2+(h+2)qs+hqs+(h+1)s^2)-(h+1)(q+s)(p+r+(q+s)(l+1))=1$, as shows a computation using that $ps-rq=det(\mu(u))=1$.
\end{proof}

\begin{corollary}\label{pyth}
Given the points $A,B,C,D$ as in the previous corollary, with $h=0$, $(t(A)+t(D),t(D)-t(A),t(B)+t(C))$ is a pythagorean triple, that is, $(t(A)+t(D))^2=(t(D)-t(A))^2+(t(B)+t(C))^2$.
\end{corollary}

As an example, look at Figure \ref{tilingN}: one has $(8^2+11^2)=(11^2-8^2)+(87+89)^2$. 

\subsection{Constructions of $SL_2$-tilings}
The following result will serve to associate $SL_2$-tilings to friezes, for any quiver; for example see exercises \ref{E6} and \ref{D7}.

\begin{proposition}\label{4sequences} (Vierfolgenansatz)
Given four two-sided sequences (that is, indexed in $\mathbb Z$) $(a_n),(b_n), (c_n), (d_n)$ in a subring $R$ of $K$, such that the following partial tiling is satisfies the $SL_2$-condition (that is each adjacent $2$ by $2$ minor is equal to $1$):
$$
\begin{array}{cccccccccc}
\ddots&\ddots&\ddots&\ddots\\
&a_{n-1}&b_{n-1}&c_{n-1}&d_{n-1}\\
&&a_{n}&b_{n}&c_{n}&d_{n}\\
&&&a_{n+1}&b_{n+1}&c_{n+1}&d_{n+1}\\
&&&&\ddots&\ddots&\ddots&\ddots
\end{array}
$$
there exist a unique tame $SL_2$-tiling $t$ over $K$ extending it to the whole plane. Moreover, its coefficients are all in $R$.
\end{proposition}

\begin{proof}
Define $\alpha_n=det(\left(\begin{array}{cc}b_n&d_n\\a_{n+1}&c_{n+1}\end{array}\right))$. If $t$ exists, then it follows from the proof of Proposition \ref{threecolumns} that $\alpha_n$ must be the linearization coefficient of the column containing $c_n$. Hence the column linearization coefficients of $t$ are all uniquely known. This implies unicity of $t$, since the given partial tiling extends uniquely to the whole plane, using the equation of Proposition \ref{threecolumns}.

This proves also the existence of $t$. Indeed, this follows from the fact that if in a 3 by 3 matrix $\left(\begin{array}{ccc}a&b&x\\c&d&y\end{array}\right)$, with respective columns $C_1,C_2,C_3$, one has $det(\left(\begin{array}{cc}a&b\\c&d\end{array}\right))=1$ and $C_3=\alpha C_2-C_1$ then $det(\left(\begin{array}{cc}b&x\\d&y\end{array}\right)=1$. Hence, the tiling constructed by applying recursively the relation of Proposition \ref{threecolumns} will satisfy the $SL_2$ condition. Moreover, by a previous remark, this construction will not interfere with the already known coefficients of the partial tiling.
\end{proof}

\begin{corollary}\label{particular4sequences}
Let $a_n$, $n\in \mathbb Z$, be a two-sided sequence over $K$. There exist a unique tame $SL_2$-tiling extending the partial tiling
$$
\begin{array}{cccccccccc}
\ddots&\ddots&\ddots&\ddots\\
&a_{n-1}&1&0&-1\\
&&a_{n}&1&0&-1\\
&&&a_{n+1}&1&0&-1\\
&&&&\ddots&\ddots&\ddots&\ddots
\end{array}
$$
This tiling is skew-symmetric with respect to the diagonal of $0$'s. The coefficients are given by the continuant polynomials, according to the rule $x=q(a_i,\ldots,a_j)$ and $y=-x$, where $x,y$ are located as indicated below.
$$  \begin{matrix}   
      0 & -1 & -a_{{i}} &\cdots&\cdots &\cdots& y\\
      1 & 0 & -1  &\ddots &&&\vdots\\
      a_{{i}} & 1 & 0 &\ddots&\ddots&&\vdots  \\[3pt]
     \vdots & \ddots&\ddots & \ddots &\ddots&\ddots&\vdots\\
     \vdots&&\ddots&\ddots&0&-1& -a_{{j}}\\
     \vdots&&&\ddots&1 & 0 & -1\\
     x & \cdots&\cdots &\cdots&a_{{j}} & 1 & 0           
     \end{matrix}
     $$
\end{corollary}

\begin{proof}
Existence and unicity follow from the previous result. Note that $a_n$ is the column linearization coefficient of the column containing $a_{n+1}$. Now, define a tiling by the rules above. Then all we have to show is that the determinant of the matrix $\left(\begin{array}{cc}q(a_i,\ldots,a_{j-1})&q(a_{i+1},\ldots,a_{j-1})\\q(a_i,\ldots,a_j))&q(a_{i+1},\ldots,a_j)\end{array}\right)$ is equal to 1. But this follows from Lemma \ref{matricecontinuants}, since $det(Q(a))=1$.
\end{proof}

\section{The other extended Dynkin diagrams}

Theorem \label{tilde A} may be extended by the methods of $SL_2$-tilings to the Dynkin diagrams of type $D$. For this one introduces some special frontiers which have a double symmetry of palindromic type. These frontiers are ultimately periodic. Moreover the geometry of the frontier ensures the presence of diagonal rays labelled by squares, or double of squares, of some integers, and the latter appear in horizontal or vertical ray; this follows from Corollary\ref{square} with $h=0$ or $h=1$. Since on a frieze of type $\tilde D$ the sequences appearing on each fork of the diagram are equal (or double one of each another, depending on the orientation), the tiling constructed as above contain all the sequences of the associated frieze. Details may be found in \cite{ARS}. An example is the $SL_2$-tiling in Figure \ref{tilingN}, which mimicks the frieze of Figure \ref{friseD7}.

For the exceptional extended Dynkin diagram, these methods do not work. However, the corresponding results have been established by quite different methods (using the Caldero-Chapoton map \cite{CC} from the cluster category into the cluster algebra) by Bernhard Keller and Sarah Scherotzke \cite{KS}. They give also new proofs for the affine case $\tilde A$ and $\tilde D$. They prove rationality of the friezes (but not $\mathbb N$-rationality).

Actually, one may work more generally by considering the initial value $a(v,0)$ as an independant variable, instead of $a(v,0)=1$ as in Section \ref{frieze}. Then the sequences of the frieze are Laurent polynomials in these variables (Laurent phenomenon of Fomin and Zelevinsky) and for the Dynkin and extended Dynkin quivers, they have nonnegative integer coefficients (for general quivers, this is conjectured: a particular case of the {\em positivity conjecture} of Fomin and Zelevinsky). They satisfy linear recurrences also, as shown for type $\tilde A$ in \cite{ARS} and in general in \cite{KS}.

Finally, note that all this works also for Cartan matrices instead of diagrams (loc. cit.).

\section{Problems and conjectures}

Th.\ref{theorem1} is proved under the assumption that the sequences on the frieze are rational. Rational sequences are exponentially bounded. Examples of graphs which are not Dynkin nor extended Dynkin diagrams show that the sequences of the associated frieze grow very fast, more than exponentially. It seems likely that Th.\ref{theorem1} is still true under the weaker assumption that the sequences are exponentially bounded.

In \cite{KS}, Keller and Scherotzke show that the friezes of type ${\mathbb E}$ are rational; we conjecture that they are actually $\mathbb N$-rational. Actually, this should be deduced from the very special form of the linear recursions satisfied by the sequences, as is shown in \cite{KS}, Section 10, together with Soittola's theorem, see Th.\ref{merge}.

In the same article \cite{KS}, the authors give linear recursions for the sequences on the friezes of extended Dynkin type, in the case where the initial values of the frieze are independent variables (so that the sequences of the frieze are Laurent polynomials in theses variables); they ask if these recursions are the shortest ones. If we set each variable to 1, we get sequences over $\mathbb N$, and linear recursions over $\mathbb Z$; it is an interesting question to know what are the shortest recursions in the numerical case: the previous example shows that they may be shorter when the variables are set to 1.

It is tempting to make also the following conjecture, inspired by the results presented here. Let $x_1x_2x_3\cdots$ be a left-infinite word on the alphabet $\{x,y\}$. Let $i,j\in \{1,2\}$. Suppose that the sequence $\mu(x_1\cdots x_n)_{ij}$ is rational (or even $\mathbb N$-rational). Then the left-infinite word above is ultimately periodic (the converse is easy to prove, cf. proof of Theorem \ref{rational}).

\section{Exercices}

\begin{exercice}
\label{ex1}
Add the missing values in Figure \ref{tilingN} so that it becomes a $SL_2$-rectangle.
\end{exercice}

\begin{exercice}
\label{exmatrix}
Show that the multiplicative monoid of nonnegative matrices in $SL_2(\mathbb Z)$ is generated by the two matrices $\mu(x)$ and $\mu(y)$. Show that it is a free monoid on theses two matrices.
\end{exercice}

\begin{exercice}
\label{}
Construct the frieze associated to the quiver with two vertices and a double edge from one to another; identify the Fibonacci numbers
\end{exercice}

\begin{exercice}
Show that the frieze of Figure \ref{friseE6} defines three sequences $a_n,b_n,c_n$ satisfying
$$
a_{n+1}=\frac{1+b_n}{a_n}, b_{n+1}=\frac{1+a_{n+1}c_n}{b_n}, c_{n+1}=\frac{1+b_{n+1}^3}{c_n},
$$
with the initial values $a_0=b_0=c_0=1$. Show that they are $\mathbb N$-rational. 

Hint: show that $a_{2n}=F_{4n-2},a_{2n+1}=2F_{4n}$, where $F_n$ are the Fibonacci numbers (with $F_0=F_1=1$), that $b_n$ satisfies the recursion $b_{n+2}=7b_{n-1}-b_{n-2}-1$, so that $b_n=2(F_{4n-3}+F_{4n-7}+\ldots+F_1)+1$, and that $c_n$ satisfies the linear recursion, associated to the polynomial $(x^4-322x^2+1)(x^{12}-322x^6+1)$, which I deduced from the results in \cite{KS}, Section 10 (see also \cite{AD} p. 2338 and table 6); note that $322=F_{12}+F_{10}$. The $\mathbb N$-rationality is then deduced by using Soittola's theorem.

\end{exercice}

\begin{exercice}
\label{}
Show that for $1\leq i<j\leq k$, $$q(a_1,\ldots,a_{j-1})q(a_{i+1},\ldots,a_k)=q(a_1,\ldots,a_{i-1})q(a_{j+1},\ldots, a_k)$$$$+q(a_1,\ldots,a_k)q(a_{i+1},\ldots,a_{j-1}).$$
See \cite{AR} Lemma 2.2.
\end{exercice}

\begin{exercice}
\label{pathex}
Prove Proposition \ref{paths}.
\end{exercice}

\begin{exercice}
\label{exquad}
Show that the quadratic forms which appear in Subsection \ref{quad} are those whose coefficients are natural numbers and with discriminant of the form $n^2 -4$, $n\in \mathbb N$. Hint: let $Q(x,y)=ax^2+exy+dy^2$ with discriminant $n^2-4$. Show that $e=n+2f$, $f\in \mathbb N$. Let $b=h+n$, $c=h$. Show that the determinant of $\left(\begin{array}{cc}a&b\\c&d\end{array}\right)$ is 1 and using Exercise \ref{exmatrix}, find a word $w$ such that $\mu(w)$ is equal to this matrix.
\end{exercice}

\begin{exercice}
\label{}
Show that a pythagorean triple $(a,b,c)$ is of the form of those appearing in Corollary \ref{pyth} if and only if $a,b,c$ are relatively prime and $c$ is even.
\end{exercice}

\begin{exercice}
\label{fringeext}
Let $F\subset {\mathbb Z}^2$ be a fringe (see \ref{pathmodel} for the definition). Let $t:F\rightarrow \mathbb Z$ (resp. $F\rightarrow \mathbb N$) be a partial $SL_2$-tiling. Show that it extends uniquely to a tame $SL_2$ tiling (resp. to an $SL_2$-tiling, necessarily tame) of the plane.
Hint: show that the linearization coefficients of $t$ are determined by the hypothesis.
\end{exercice}

\begin{exercice}
\label{}
Deduce from Exercise \ref{fringeext} that if a frieze $a(v,n)$ of type $\tilde{\mathbb A}$ has values in $\mathbb Z$ (resp. $\mathbb N$) for two successive $n$ and any $v$, then all its values are in $\mathbb Z$ (resp. $\mathbb N$). Show that if it has values in $\mathbb N$, then it is $\mathbb N$-rational.
\end{exercice}

\begin{exercice}
\label{E6}
Use Proposition \ref{4sequences} and the frieze of Figure \ref{friseE6} to show the existence of the integral $SL_2$-tiling of Figure \ref{tilingE6}.
\end{exercice}

\begin{figure}
$$
\begin{array}{ccccccccccccccccccccccc}
10&7&4&1&0&-1&-2\\
7&5&3&1&1&0&-1&-2\\
4&3&2&1&2&1&0&-1&-10\\
1&1&1^2&1&3&2&1&0&-1&-13\\
6&7&8&3^2&28&19&10&1&0&-1&-68\\
77&90&103&116&19^2&245&129&13&1&0&-1\\
5230&6113&6996&7879&24520&129^2&8762&883&68&1&0
\end{array}
$$
\caption{A tiling associated to the diagram $\tilde{\mathbb E}_6$}\label{tilingE6}
\end{figure}

\begin{exercice}
\label{D7}
As in the previous exercise, do the same for the frieze of Figure \ref{friseD7}, extended in both directions, and the tiling below.
$$
\begin{array}{llllllllllllllllllllllllll}
&&&&&&&&&& & & & & \\
&&&&&&&\bf 20^2&61&27&74&195&316&437 \\
&&&&&&\bf 109&\bf 59&\bf 3^2&4&11&29&47&65 \\
&&&&&&\bf 24&\bf 13&\bf 2&\bf 1^2&3&8&13&18& \\
&&&&&&\bf 11&\bf 6&\bf  1& \bf 1&\bf 2^2&11&18&25& \\
&&&&&&\bf 3^2&\bf 5&\bf 1&\bf 2&\bf 9&\bf 5^2&41&57& \\
&&&&&&7&\bf 2^2&\bf 1&\bf 3&\bf 14&\bf 39&\bf 8^2&89& \\
&&&&&&5&3&\bf 1^2&\bf 4&\bf 19&\bf 53&\bf 87&\bf 11^2& \\
&&&&&&28&17&6&\bf 5^2&\bf 119&\bf 332&\bf 545&\bf 758& \\
&&&&&&135&82&29&121&\bf 24^2&\bf 1607&\bf 2368&\bf 3669& \\
&&&&&&&&&&&&&& \\
\end{array}
$$
Hint: the boldfaced numbers appear on the frieze. It may be useful to extend the Proposition \ref{4sequences} to more than 4 sequences, here 6.
\end{exercice}

\begin{exercice}
\label{}
Show that the following array extends uniquely to a tame $SL_2$-tiling of the plane, with positive integers below the diagonal of 1's.
$$
\begin{array}{lllllllllllllllllll}
1&0\\
&1&0\\
&&1&0\\
&&&1&0\\
&&&&1&0\\
&&&&&1&0\\
&&&1&1&1&1&0\\
&&1&\bf 1&&&&1&\bf 0\\
&&1&&&&&&1&0\\
1&1&1&&&&&&&1&0\\
1&&&&&&&&&&1&0\\  
\bf 1&&&&&&&&x&&&1&0
\end{array}
$$
Hint: determine the column linearization coefficients similarly to Corollary \ref{lincoeff}; then determine the values immediately under the diagonal of 1's by using a remrk in the proof of Corollary \ref{particular4sequences}. Find a formula for the coefficients, similar to Corollary \ref{t(P)=mu(w)}; for example, for the point with value $x$ in the array, define its word as $w=yyxxyyx$, by projecting as indicated by the boldfaced numbers, and then $x=\mu(w)_{22}$.

State and prove a result generalizing this example.
\end{exercice}

 \end{document}